\documentclass[11pt]{amsart}
\usepackage{lmodern}
\usepackage{amsmath, amsthm, amssymb, amsfonts}
\usepackage[normalem]{ulem}
\usepackage{hyperref}

\usepackage{mathrsfs}

\usepackage{verbatim} 
\usepackage{longtable}

\usepackage{mathtools}

\usepackage{tikz}
\usetikzlibrary{decorations.pathmorphing}
\tikzset{snake it/.style={decorate, decoration=snake}}

\usepackage{caption}

\usepackage{tikz-cd}
\usetikzlibrary{arrows}

\theoremstyle{plain}
\newtheorem{thm}{Theorem}[section]
\newtheorem{cor}[thm]{Corollary}
\newtheorem{lem}[thm]{Lemma}
\newtheorem{prop}[thm]{Proposition}
\newtheorem{conj}[thm]{Conjecture}
\newtheorem{question}[thm]{Question}

\theoremstyle{definition}

\newtheorem{example}[thm]{Example}

\theoremstyle{remark}
\newtheorem{rmk}[thm]{Remark}

\newcommand{\BC}{{\mathbb{C}}}

\newcommand{\BG}{{\mathbb{G}}}

\newcommand{\BL}{{\mathbb{L}}}

\newcommand{\BN}{{\mathbb{N}}}

\newcommand{\BQ}{{\mathbb{Q}}}

\newcommand{\BZ}{{\mathbb{Z}}}

\newcommand{\CA}{{\mathcal A}}

\newcommand{\CC}{{\mathcal C}}
\newcommand{\CD}{{\mathcal D}}
\newcommand{\CE}{{\mathcal E}}
\newcommand{\CF}{{\mathcal F}}
\newcommand{\CG}{{\mathcal G}}
\newcommand{\CH}{{\mathcal H}}

\newcommand{\CK}{{\mathcal K}}
\newcommand{\CL}{{\mathcal L}}
\newcommand{\CM}{{\mathcal M}}

\newcommand{\CO}{{\mathcal O}}
\newcommand{\CP}{{\mathcal P}}
\newcommand{\CQ}{{\mathcal Q}}

\newcommand{\CV}{{\mathcal V}}
\newcommand{\CW}{{\mathcal W}}

\DeclareFontFamily{OT1}{rsfs}{}
\DeclareFontShape{OT1}{rsfs}{n}{it}{<-> rsfs10}{}
\DeclareMathAlphabet{\curly}{OT1}{rsfs}{n}{it}

\usepackage{tikz}
\usepackage{lmodern}
\usetikzlibrary{decorations.pathmorphing}

\begin{document}
\title[Fourier--Mukai transforms and the decomposition theorem]{Fourier--Mukai transforms and the decomposition theorem for integrable systems}
\date{\today}

\author[D. Maulik]{Davesh Maulik}
\address{Massachusetts Institute of Technology}
\email{maulik@mit.edu}

\author[J. Shen]{Junliang Shen}
\address{Yale University}
\email{junliang.shen@yale.edu}

\author[Q. Yin]{Qizheng Yin}
\address{Peking University}
\email{qizheng@math.pku.edu.cn}

\begin{abstract}
We study the interplay between the Fourier--Mukai transform and the decomposition theorem for an integrable system $\pi:  M \rightarrow B$.  Our main conjecture is that the Fourier--Mukai transform of sheaves of K\"ahler differentials, after restriction to the formal neighborhood of the zero section, are quantized by the Hodge modules arising in the decomposition theorem for $\pi$. For an integrable system, our formulation unifies the Fourier--Mukai calculation of the structure sheaf by Arinkin--Fedorov, the theorem of the higher direct images by Matsushita, and the ``perverse = Hodge'' identity by the second and the third authors. 

As evidence, we show that these Fourier--Mukai images are Cohen--Macaulay sheaves with middle-dimensional support on the relative Picard space, with support governed by the higher discriminants of the integrable system. We also prove the conjecture for smooth integrable systems and certain 2-dimensional families with nodal singular fibers. Finally, we sketch the proof when cuspidal fibers appear. 




\end{abstract}

\maketitle

\setcounter{tocdepth}{1} 

\tableofcontents
\setcounter{section}{-1}

\section{Introduction}

\subsection{A Fourier--Mukai/Decomposition  correspondence}\label{Sec0.1}
The purpose of this paper is to formulate and explore a correspondence between two geometric structures associated with an integrable system --- the Fourier--Mukai transform \cite{Mukai} and the decomposition theorem \cite{BBD}.

Roughly speaking, a geometric model of an (algebraically completely) integrable system is a Lagrangian fibration $\pi_M: M \to B$ with $M$ a symplectic variety.  Motivated by mirror symmetry and the geometric Langlands correspondence, in many interesting cases the fibration $M \to B$ is conjectured to admit a ``dual'' Lagrangian fibration $\check{M} \to B$ which extends the dual abelian scheme associated with the nonsingular fibers of $\pi_M$. Moreover, the two ambient spaces $M$ and $\check{M}$ are expected to share the same derived category of coherent sheaves via a Fourier--Mukai transform
\begin{equation}\label{FM_conj}
\phi_{\mathrm{FM}}: D\mathrm{Coh}(M) \to D\mathrm{Coh}(\check{M})
\end{equation}
extending the classical Fourier--Mukai transform \cite{Mukai} of dual abelian schemes. We are interested in the Fourier--Mukai images of the locally free sheaves $\Omega^k_M$ of K\"ahler differentials on~$M$:
\begin{equation}\label{FM0}
\phi_{\mathrm{FM}}(\Omega_M^k) \in D^b\mathrm{Coh}(\check{M});
\end{equation}
when $M$ is hyper-K\"ahler, $\Omega_M^k$ are natural \emph{hyper-holomorphic bundles} and (\ref{FM0}) are their homological mirrors.\footnote{See Section \ref{HMS} for further discussions on this aspect.} Our proposal is that the objects (\ref{FM0}) are closely related to the decomposition theorem \cite{BBD} for the original integrable system $\pi_M: M \to B$:
\[
R\pi_{M*} \BQ_M[ \dim M /2] \simeq \bigoplus_{i=0}^{\dim M} P_k [-k], \quad P_k= {^\mathfrak{p}\CH}^k( R\pi_* \BQ_M[\dim M/2] ).
\]
Here $P_k$ can either be viewed as a perverse sheaf or a holonomic $\CD_B$-module on $B$. 

The rough version of the Fourier--Mukai/Decomposition correspondence is:
\begin{equation}\label{FM/D0}
    \phi_{\mathrm{FM}}(\Omega_M^k) \textup{ ``$\simeq$'' } P_k, \quad \textup{for all } k \in \BZ.
\end{equation}
Clearly (\ref{FM/D0}) does not make sense if we interpret it in a na\"ive way; the left-hand side is a coherent object on the symplectic variety $\check{M}$ while the right-hand side lies in the abelian category of $\CD_B$-modules. Ultimately, we are able to modify both sides of (\ref{FM/D0}) to get a mathematically precise formulation; we refer to Conjecture \ref{main_conj} for the statement. 

Before diving into more technical details, we briefly summarize some main ingredients and ideas. By Saito's theory of mixed Hodge modules \cite{S1,S2}, the holonomic $\CD_B$-module $P_k$ admits a natural good filtration. Taking the associated graded object, we obtain a coherent sheaf $\mathrm{gr}(P_k)$ on the symplectic variety $T^*B$, which can be viewed as the ``classical limit'' of~$P_k$. On the other hand, in the case when $\check{M} \to B$ admits a section, the normal bundle of the section is identified with the cotangent bundle by the symplectic form of $\check{M}$. Then, the modified version of (\ref{FM/D0}) is an isomorphism which holds in the common formal neighborhood $\hat{B}$ of $B$ in both symplectic varieties $\check{M}$ and $T^*B$, between two coherent objects:
\[
\phi_{\mathrm{FM}}(\Omega_M^k)|_{\hat{B}} \quad \mathrm{and} \quad \mathrm{gr}(P_k)|_{\hat{B}}.
\]
In other words, the Hodge module $P_k$ ``quantizes'' the restriction of $\phi_{\mathrm{FM}}(\Omega_M^k)$ to $\hat{B}$.

\subsection{Fourier--Mukai}\label{Sec0.2}
Throughout, we work over the complex numbers $\BC$. Modifying Definition $2.6$ of \cite{AF}, we say that \[
\pi_M: (M, \sigma) \to B, \quad \dim M = 2n
\]
is an \emph{integrable system} if $(M, \sigma)$ is a nonsingular holomorphic symplectic variety carrying a Lagrangian fibration
\begin{equation*}\label{Lag}
\pi_M: M \to B
\end{equation*}
such that $\pi_M$ is projective, the base $B$ is nonsingular, and it further satisfies that:
\begin{enumerate}
    \item[(i)] every geometric fiber of $\pi_M$ is integral, and
    \item[(ii)] the fibration $\pi_M$ admits a section $s_B: B \rightarrow M$.\footnote{In the setting of Arinkin--Fedorov, an integrable system is not required to have a section. Here we impose a stronger assumption.}
\end{enumerate}
Hitchin and Beauville--Mukai systems (Example \ref{ex1.1}) provide rich examples of integrable systems. By \cite[Corollary 2.7]{AF}, the identity component of the relative Picard space is a smooth group (algebraic) space $\pi_P: P(:=\mathrm{Pic}^0(M/B)) \to B$ over the base $B$ which admits a universal line bundle 
\[
\CL \rightarrow M\times_B P.
\]
We may assume that $\CL$ is trivialized along the $0$-sections $0_B: B\hookrightarrow P$ and $s_B: B \hookrightarrow M$. 

Our main character is Arinkin--Fedorov's (partial) Fourier--Mukai transform \cite{AF}:
\begin{equation}\label{PFM}
\phi_{\mathrm{FM}}: D^b\mathrm{Coh}(M) \rightarrow D^b\mathrm{Coh}(P),\quad \CE \mapsto R{q_{P*}}(q_M^* \CE \otimes \CL^\vee) \otimes \pi_P^*\omega_B^\vee[n]
\end{equation}
where the $q_{(-)}$ are the natural projections from $M \times_BP$. In general, the existence of a dual Lagrangian fibration $\check{M} \to B$ and the Fourier--Mukai equivalence (\ref{FM_conj}) is conjectural and wide open for most cases. If they exist, $\check{M}$ is expected to be a holomorphic symplectic partial compactification of $P \to B$ and (\ref{FM_conj}) should extend (\ref{PFM}). Nevertheless, since we are only interested in the Fourier--Mukai image near the section 
\[
B \subset P \,(\subset \check{M})
\]as we explained at the end of Section \ref{Sec0.1}, the rigorously defined partial transform (\ref{PFM}) is enough for our purpose.

Concerning the decomposition side, recall that by \cite{S1,S2}, the perverse sheaves $P_k$ are naturally Hodge modules on $B$, whose associated graded objects induce coherent sheaves
\[
\mathrm{gr}(P_k) \in \mathrm{Coh}(T^*B).
\]
As we will see in Proposition \ref{prop1.4}, a choice of $\pi_M$-relatively ample bundle $\Theta$ on $M$ induces an identification 
$\widehat{\kappa}_{\Theta}$ between the formal neighborhoods of $B$ inside $P$ and $T^*B$ respectively, each of which we denote by $\hat{B}$:
\[
\hat{B} \subset P, \quad \hat{B} \subset T^*B.
\]

Using this identification, we formulate the Fourier--Mukai/Decomposition correspondence as follows; see Conjecture \ref{main_conj2} for a more precise formulation.

\begin{conj}[Fourier--Mukai/Decomposition correspondence]\label{main_conj}
For an integrable system $\pi_M: (M, \sigma) \to B$, we have
\begin{equation}\label{main}
\phi_{\mathrm{FM}}(\Omega_M^k)|_{\hat{B}} \simeq \mathrm{gr}(P_k)|_{\hat{B}} \in D^b\mathrm{Coh}(\hat{B}).
\end{equation}
\end{conj}

Conjecture \ref{main_conj} is hence a coherent--constructible correspondence for an integrable system $\pi_M: M \to B$, which connects the coherent object $\Omega_M^k$ on $M$ to the constructible object~$P_k$ on $B$.

\subsection{Lagrangian Cohen--Macaulay sheaves}

Our first result is a consistency check for Conjecture~\ref{main_conj}, in which we check both sides are given by Cohen-Macaulay sheaves and match the reduced (Lagrangian) supports under~(\ref{1223}).

For the Hodge module $P_k$, these are known already as follows.

 We denote by
\[
\mathrm{supp}^{\mathrm{red}}( \mathrm{gr}(P_k)) \subset T^*B,
\]
the support of $\mathrm{gr}(P_k)$ endowed with the reduced scheme structure; it
is a conical Lagrangian depending only on the underlying perverse sheaf $P_k$. More precisely, it is the singular support~$\mathrm{SS}(P_k)$ of the perverse sheaf $P_k$, and is  described by the \emph{higher discriminants} of the morphism $\pi_M: M \to B$ \cite{disc}. For any $i\geq 0$, Migliorini--Shende introduced a higher discriminant $\Delta^i(\pi_M) \subset B$ determined by the topology of $\pi_M: M \to B$, which satisfies
\[
\mathrm{codim}_{B} \Delta^i(\pi_M)  \geq i.
\]
They further proved for any $k$ that
\begin{equation}\label{Lambda}
\mathrm{SS}(P_k) \subset \Lambda: =  \bigcup_i \bigcup_{Z_i} \overline{T^*_{Z_i}B}.
\end{equation}
Here $Z_i$ runs through purely $i$-codimensional irreducible components of $\Delta^i(\pi_M)$, and $\overline{T^*_{Z_i}B}$ stands for the closure of the conormal bundle of the nonsingular locus of $Z_i$. Clearly the conical Lagrangian $\Lambda \subset T^*B$ only depends on the topology of $\pi_M: M \to B$.

\begin{thm}[Saito, Migliorini--Shende]\label{gr(P)}
For any $k$ the coherent sheaf $\mathrm{gr}(P_k)$ is a Cohen--Macaulay sheaf of dimension $n$ with 
\[
\mathrm{supp}^{\mathrm{red}}(\mathrm{gr}(P_k)) \subset \Lambda.
\]
\end{thm}

The second part of the theorem follows from \cite{disc} as discussed above; the first part is a property of any mixed Hodge module on a nonsingular variety \cite{S1}.

We first establish the counterpart of Theorem \ref{gr(P)} on the Fourier--Mukai side. In Section~\ref{Sec2.4}, we define a closed subset $\Lambda' \subset P$ using the higher discriminants $\Delta^i(\pi_M)$, parallel to $\Lambda \subset T^*B$; by construction, the restriction of $\Lambda'$ to the formal neighborhood of $B$ recovers the conical Lagrangian $\Lambda \subset T^*B$ by Corollary \ref{cor2.4}:
\begin{equation}\label{1223}
\Lambda' |_{\hat{B}} \simeq \Lambda |_{\hat{B}} \subset \hat{B}.
\end{equation}

\begin{thm}\label{thm0.3}
The object 
\[
\phi_{\mathrm{FM}}(\Omega_M^k) \in D^b\mathrm{Coh}(P)
\]
is a Cohen--Macaulay sheaf of dimension $n$ concentrated in degree 0 with
\begin{equation*}
\mathrm{supp}^{\mathrm{red}}(\phi_{\mathrm{FM}}(\Omega_M^k))|_{\hat{B}} \subset {\Lambda'}|_{\hat{B}}.
\end{equation*}
\end{thm}

In general the Fourier--Mukai image of a coherent sheaf is a complex concentrated in degrees $[-n,0
]$. Therefore the proof of Theorem \ref{thm0.3} relies on properties of the sheaves~$\Omega^k_M$.  As a consequence of Theorem \ref{thm0.3}, Conjecture \ref{main_conj} is reduced to an isomorphism of two Cohen--Macaulay sheaves on $\hat{B}$.

\subsection{Smooth and elliptic fibrations}
As further evidence for Conjecture \ref{main_conj}, we prove it 
in the following cases.

First, we verify the conjecture for smooth Lagrangian fibrations. 

\begin{thm}\label{thm0.4}
Conjecture \ref{main_conj} holds if $\pi_M$ is smooth.
\end{thm}

Next, we prove the conjecture when $M$ is an elliptic surface over a non-proper curve $B$ which has at worst nodal fibers. This is the first nontrivial case where singular fibers appear.

\begin{thm}\label{thm0.5}
Conjecture \ref{main_conj} holds if $M$ is an elliptic surface over a non-proper curve $B$ with nodal singular fibers.
\end{thm}

The strategy for the proof is to view both sheaves as iterated extensions of certain sheaves and then match the terms as well as the extensions.  This matching involves a delicate argument using Deligne's canonical extension.
In Section \ref{sec6.7}, we also sketch a proof of Conjecture \ref{main_conj} for non-proper 2-dimensional $M$ when cuspidal fibers appear. The argument in this setting follows the same strategy as for Theorem \ref{thm0.5} but is even more complicated due to the cuspidal fibers, so we leave the details to the interested reader.

\subsection{Relations to other work and directions}

\subsubsection{The perverse--Hodge symmetry}

The motivation for Conjecture \ref{main_conj} is an effort to understand and categorify the perverse--Hodge symmetry for Lagrangian fibrations \cite{SY}; see also \cite{HLSY, FSY, HM, SY2}.

For a Lagrangian fibration $\pi_M: M \to B$ with $M$ a $2n$-dimensional nonsingular compact irreducible symplectic variety, a perverse--Hodge symmetry was found in \cite{SY}:
\begin{equation}\label{P=F}
    \dim H^{i-n}(B, P_{k}) = \dim H^{k,i}(M),
\end{equation}
whose proof relies heavily on the global geometry of compact holomorphic symlectic/hyper-K\"ahler manifolds.

A categorification of (\ref{P=F}) was proposed recently in \cite{SY2} using the natural lift of each $P_k$ to a Hodge module on $B$. By taking graded pieces of the de Rham complex of $P_k$, we obtain a bounded complex of coherent sheaves:
\[
\mathrm{gr}_{i}\mathrm{DR}(P_k) \in D^b\mathrm{Coh}(B).
\]
The main conjecture of \cite{SY2} is:
\begin{equation}\label{P=F2}
    \mathrm{gr}_{-i}\mathrm{DR}(P_{k})[n-k] \simeq  \mathrm{gr}_{-k}\mathrm{DR}(P_{i})[n-i] \in D^b\mathrm{Coh}(B).
\end{equation}
A mysterious feature of (\ref{P=F2}) is that, unlike (\ref{P=F}), it does not need the compactness assumption of $M$. When $M$ is indeed compact, it was explained in \cite[Section 4]{SY2} that (\ref{P=F2}) recovers (\ref{P=F}).

Our attempt here goes back to the original form of (\ref{P=F}) where we view the right-hand side as cohomology of the vector bundles $\Omega_M^k$. This suggests that there may exist a coherent--constructible correspondence connecting $\Omega^k_M$ and $P_k$ directly, which ideally should hold for not necessarily compact $M$ admitting a Lagrangian fibration. Our conjecture realizes this for integrable systems, as defined in Section \ref{Sec0.2}, using the Fourier--Mukai transform.

The following proposition proves the compatibility between Conjecture \ref{main_conj}, the perverse--Hodge symmetry (\ref{P=F}), and its categorification (\ref{P=F2}) on the base $B$; it gives further evidence for Conjecture \ref{main_conj}.

\begin{prop}\label{prop0.6}
Let $\pi_M: M \to B$ be an integrable system for which (\ref{P=F2}) holds. Then we have 
\begin{equation}\label{P=F4}
L0_B^* \,\phi_{\mathrm{FM}}(\Omega_M^k) = L0_B^*\, \mathrm{gr}(P_k) \in D^b\mathrm{Coh}(B),
\end{equation}
where $0_B$ denotes both the closed embeddings of the $0$-sections in $P$ and $T^*B$. In particular, if $B$ is projective, then we have
\begin{equation}\label{P=F3}
H^i(B, L0_B^*\,\phi_{\mathrm{FM}}(\Omega_M^k) ) = H^i(B, L0_B^*\,\mathrm{gr}(P_k)).
\end{equation}
\end{prop}

In fact as we will see in its proof, (\ref{P=F3}) is equivalent to (\ref{P=F}). Hence Conjecture \ref{main_conj} can be viewed as extending the perverse--Hodge symmetry from the $0$-section $B$ to a neighborhood~$\hat{B}$. We note that even in the case where $\pi_M: M \to B$ is smooth, in general the sheaves~$\phi_{\mathrm{FM}}(\Omega_M^k)$ and $\mathrm{gr}(P_k)$ are not scheme-theoretically supported on the $0$-sections. When the integrable system has singular fibers, the formal neighborhood $\hat{B}$ contains much richer geometry than the $0$-section $B$.  
One advantage of Conjecture \ref{main_conj} compared to (\ref{P=F2}) and (\ref{P=F4}) is that involves matching Cohen--Macaulay sheaves, instead of matching complexes in the derived category.
It would be interesting to find an enhancement of Conjecture \ref{main_conj} that would imply the full 
statement of (\ref{P=F2}).

\subsubsection{The geometric Langlands correspondence}
A more speculative direction for studying the Fourier--Mukai transform of the sheaf of K\"ahler differentials comes via the geometric Langlands correspondence.  

For a nonsingular curve $C$ and a reductive group $G$, the geometric Langlands correspondence predicts an equivalence of categories
\begin{equation}\label{GLC}
\mathrm{GLC}: \textup{``}D{\mathrm{QCoh}}\textup{''}(\CM_{\mathrm{dR},G}) \xrightarrow{\simeq} D(\mathrm{Bun}_{^LG}, \CD)
\end{equation}
respecting various extra structures. Here $\CM_{\mathrm{dR},G}$ stands for the de Rham moduli stack parameterizing $G$-local systems on $C$, $\mathrm{Bun}_{^LG}$ is the moduli stack of principal $^LG$-bundles on $C$, $D(-, \CD)$ is the derived category of $D$-modules, and the quotation marks means that ``the category of quasi-coherent sheaves'' needs to be modified; we refer to \cite{AG} for the precise statement. 

The ``classical limit'' of (\ref{GLC}) is expected to be induced by a Fourier--Mukai transform
\[
\phi_{\mathrm{FM}}: D\mathrm{QCoh}(\CM_{\mathrm{Higgs},G}) \xrightarrow{\simeq} D\mathrm{QCoh}(\CM_{\mathrm{Higgs},{^LG}})
\]
for the dual Lagrangian fibrations given by the Hitchin systems associated with $G$ and its Langlands dual $^LG$.\footnote{The Fourier--Mukai $\phi_{\mathrm{FM}}$ here should differ from the Fourier--Mukai we use in this paper by a shift, as we expect that a skyscraper sheaf is sent to a Hecke eigensheaf which is quantized by a holonomic $D$-module. However, for notational convenience we ignore the shift for the discussion here.}

In our setting, if we consider the (derived) exterior power
\[
\Omega^k: = \wedge^k \BL 
\]
of the cotangent complex $\BL$ on $\CM_{\mathrm{Higgs},G}$,
this object can be lifted to the left-hand side of~(\ref{GLC})). Therefore, we might expect to find an object $\CQ_k$ in $D(\mathrm{Bun}_{^LG}, \CD)$ which quantizes $\phi_{\mathrm{FM}}(\Omega^k)$. 

\begin{question}
Can we describe the object $\CQ_k \in D(\mathrm{Bun}_{^LG}, \CD)$ quantizing $\phi_{\mathrm{FM}}(\Omega^k)$?
\end{question}

Our proposal provides a strange, partial answer to the question. To avoid stacky issues, we focus on the open subset $A^\circ \subset A$ of the Hitchin base over which both Hitchin fibrations 
\[
h_G: \CM_{\mathrm{Higgs}, G} \to A, \quad h_{^LG}: \CM_{\mathrm{Higgs}, {^LG}} \to A
\]
are proper with integral fibers. Our main conjecture predicts that, the object $\phi_{\mathrm{FM}}(\Omega^k)$ is quantized by a single (shifted) $D$-module $P_k$ obtained from the decomposition theorem associated with \[
h_{^LG}|_{A^\circ}: \CM_{\mathrm{Higgs},^LG}|_{A^\circ} \to A^\circ
\]
if we restrict to a formal neighborhood of the Kostant section.\footnote{According to the Ng\^o support theorem \cite{Ngo}, the decomposition theorem associated with the Hitchin system~$h_G$ coincides with that associated with $h_{^LG}$, since they are both governed by intermediate extensions of the local systems given by the smooth fibers.} In particular this suggests that $\CQ_k$ is a (shifted) holonomic $D$-module.

\subsubsection{Hyper-K\"ahler mirror symmetry}\label{HMS} 

The transform $\phi_{\mathrm{FM}}(\Omega_M^k)$ can also be studied from the perspective of hyper-K\"ahler mirror symmetry and $S$-duality \cite{KW}.

When $M$ is hyper-K\"ahler, homological mirror symmetry is expected to interchange two types of \emph{branes} on $M$ and its mirror $M^\vee$ respectively; it predicts that ``BBB'' branes on~$M$ is sent to ``BAA'' branes on $M^\vee$ via the Fourier--Mukai transform. Roughly, ``A'' stands for the Lagrangian condition and ``B'' stands for the holomorphic condition; ``BBB'' and ``BAA'' indicate the corresponding conditions with respect to the three complex structures. In particular, this suggests that 
\[
\phi_{\mathrm{FM}}( \textup{hyper-holomorphic bundle} )
\]
should be a sheaf supported on a complex Lagrangian. The main theme of this paper concerns a particular class of hyper-holomorphic bundles $\Omega_M^k$. Our main proposal gives a local description of the mirror to $\Omega_M^k$ in terms of the decomposition theorem and Hodge modules.

We note that for Hitchin systems, recently Hausel and Hitchin \cite{HH, H_ICM} studied another interesting class of hyper-holomorphic bundles given by the universal family of the Higgs bundles, whose mirrors are supported on the upward flows of certain very stable Higgs bundles.


\subsection{Acknowledgement}
We would like to thank Conan Naichung Leung, Ivan Losev, Tony Pantev, and Christian Schnell for their interest and for very helpful discussions. J.S. was supported by the NSF grant DMS-2134315. Q.Y. was supported by the NSFC grants 11831013 and 11890661.

\section{Symmetries of integrable systems}

\subsection{Overview}
In this section, we introduce two groups $A$ and $P$ associated with an integrable system $\pi_M: M\to B$ which encode its symmetries. The group scheme $A$ acts on $M$ directly, and the group space $P$ induces the Fourier--Mukai transform. They are connected by a relative polarization (\ref{uni}), which further identifies the formal neighborhood of the $0$-sections $B \subset P$ and $B \subset T^*B$.

\subsection{Integrable systems}
We recall from Section \ref{Sec0.2} that an integrable system $\pi_M: M \to B$ is a Lagrangian fibration associated with a nonsingular holomorphic symplectic variety~$(M,\sigma)$ satisfying that the base $B$ is nonsingular, the map $\pi_M$ is projective and surjective, the geometric fibers are integral, and $\pi_M$ admits a section $s_B: B\to M$.

\begin{example}\label{ex1.1}
Compactified Jacobian fibrations associated with integral curves in holomorphic symplectic surfaces provide a large class of interesting examples of integrable systems. 

If we take $(S,L)$ to be a polarized abelian or $K3$ surface, we consider the open subset $U \subset |L|$ of the linear system parameterizing integral curves with $\CC \to U$ the universal family. Then the compactified Jacobian fibration
\begin{equation}\label{Jacobian}
\pi_J: \overline{J}_\CC \to U
\end{equation}
parameterizing rank 1 torsion-free sheaves $\CW$ on geometric fibers $\CC_u$ of $\CC \to U$ with Euler characteristic 
\[
\chi(\CC_u, \CW) = \chi (\CC_u, \CO_{\CC_u}), \quad  u \in U
\]
gives rise to an integrable system. The section $s_B: U \to \overline{J}_\CC$ is given by the trivial line bundles:
\[
u \in U \mapsto [\CO_{\CC_u}] \in \overline{J}_\CC.
\]
The symplectic form of $\overline{J}_\CC$ is the restriction of the symplectic form on the regular locus of the moduli of semistable 1-dimensional coherent sheave on $S$ supported on $L$. Such a construction is referred to as the Beauville--Mukai system \cite{Be}.

The same construction replacing $S$ by a the non-compact symplectic surface $T^*C$ given by the total cotangent bundle of a curve $C$ of genus $\geq 2$ then recovers examples of Hitchin systems. \qed
\end{example}

Let $M^{\mathrm{sm}} \subset M$ be the smooth locus of $\pi_M$, \emph{i.e.}, the maximal open subscheme of $M$ such that the restriction of $\pi_M$ is smooth. Then the section
\[
s_B(B) \simeq B \subset M
\]
lies in $M^{\mathrm{sm}}$.

When the integrable system is given by the compactified Jacobian fibration (\ref{Jacobian}) as in Example \ref{ex1.1}, the smooth locus $\overline{J}_C^{\mathrm{sm}}$ is exactly the relative degree 0 Picard scheme  
\[
\mathrm{Pic}^0(\CC/B) \to B
\]
parameterizing degree 0 line bundles on geometric fibers; it is a group scheme over $B$ which acts naturally on $\overline{J}_C$.

This is generalized to any integrable system by Arinkin--Fedorov \cite{AF}.

\begin{prop}[{\cite[Proposition 8.7]{AF}}]\label{prop1.2}
 The smooth locus of an integrable system is a group scheme $A:= M^{\mathrm{sm}}$ over $B$. It admits a natural action on $M$
 \[
 \mu: A \times_B M \to M
 \]
 preserving the smooth locus.
\end{prop}

\begin{proof}
For an integrable system $\pi_M: M\to B$, Arinkin and Fedorov constructed in \cite[Section~8]{AF} a $B$-group scheme $A$ from the automorphism group scheme $\mathrm{Aut}_B(X)$, so that:
\begin{enumerate}
    \item[(i)] $A$ acts naturally on $M$ preserving the smooth locus, and
    \item[(ii)] $M^{\mathrm{sm}}$ is an $A$-torsor.
\end{enumerate}
In particular, when $\pi_M$ admits a section $s_B$, the $A$-torsor $M^{\mathrm{sm}}$ is canonically identified with~$A$ via
\[
\mu\circ(\mathrm{id}_A \times s_B): A \rightarrow A\times_B M \to M, \quad \mathrm{Im}( \mu\circ(\mathrm{id}_A \times s_B) ) = M^{\mathrm{sm}} \subset M; 
\]
see the proof of \cite[Proposition 8.7]{AF}. 
\end{proof}

We conclude this section by recalling Ng\^o's $\delta$-inequality \cite{Ngo, AF}. For every closed point $b \in B$, the group $A_b$ admits the Chevalley decomposition
\[
1 \to R_b \to A_b \to H_b \to 1 
\]
with $R_b \subset A_b$ the maximal affine subgroup and $H_b$ an abelian variety. This defines a semicontinuous function
\[
\delta: B \to \BN, \quad \delta(b): = \dim R_b
\]
which calculates the dimension of the affine part. 

The following proposition is a consequence of the existence of a symplectic form.

\begin{prop}[Ng\^o; see {\cite[Proposition 8.9]{AF}}]\label{prop1.3}
For an integrable system $\pi_M: M \to B$ with~$A$ the associated group scheme as in Proposition \ref{prop1.2}, we have
\[
\mathrm{codim}_B\{b\in B\,|\, \delta(A_b) \geq i \} \geq i.
\]
\end{prop}



\subsection{Picard}
We introduce here another group over $B$ closely related to $A$, the Picard space, which will appear in the Fourier--Mukai transform.

The (relative) Picard stack represents the functor
\[
\mathcal{P}ic(M/B): B\textrm{-schemes} \to \mathrm{groupoids}
\]
sending a $B$-scheme $S$ to the groupoid of line bundles over $M \times_BS$. The section $s_B: B \to M$ trivializes the $\BG_m$-gerbe:
\[
\mathcal{P}ic(M/B) = \mathrm{Pic}(M/B) \times B\BG_m.
\]
Here $\mathrm{Pic}(M/B)$ is a $B$-group algebraic space whose restriction to a geometric point $b \in B$ recovers the Picard scheme $\mathrm{Pic}(M_s)$. In particular, the algebraic space $\mathrm{Pic}(M/B)$ carries a universal line bundle
\begin{equation}\label{uni}
\CL \rightarrow M\times_B \mathrm{Pic}(M/B).
\end{equation}
The identity component $P \subset \mathrm{Pic}(M/B)$ is also a group space over $B$, and we obtain a universal line bundle $\CL \to M\times_B P$ by the restriction of (\ref{uni}), which satisfies
\[
\CL|_{M\times_B 0_B(B)} \simeq \CO_M.
\]
By the choice of splitting, the universal bundle $\CL$ on $M\times_BP$ is normalized to satisfy the condition 
\[
\CL|_{s_B(B)\times_BP} \simeq \CO_P.
\]
We will always work with this normalized universal family $\CL$.

The two groups $A$ and $P$ over $B$ are closely related. If we choose a $\pi_M$-relatively ample line bundle 
\begin{equation}\label{choice}
\Theta \in \mathrm{Pic}(M),
\end{equation}
we obtain a Poincar\'e line bundle
\[
\CP_\Theta:= \mu^*\Theta \otimes \mathrm{pr}_M^* \Theta
\]
over $A\times_BM$, which further induces a morphism 
\[
\kappa_\Theta: A \rightarrow \mathrm{Pic}(M/B).
\]
By \cite[Corollary 7.7]{AF}, the morphism $\kappa_\Theta$ is \'etale. Since $A = M^{\mathrm{sm}} \subset M$ is irreducible, the image of $\kappa_\Theta$ has to lie in the identity component $P \subset \mathrm{Pic}(M/B)$. We thus obtain the \'etale homomorphism of $B$-groups
\[
\kappa_\Theta: A \to P.
\]

The following proposition serves as the foundation for Conjecture \ref{main_conj} which identifies the formal neighborhoods of $B$ inside $P$ and $T^*B$.

\begin{prop}\label{prop1.4}
The $\pi_M$-relatively ample line bundle $\Theta$ in \eqref{choice} induce an isomorphism
\[
\hat{\kappa}_\Theta:  \hat{B}_P \xrightarrow{\simeq} \hat{B}_{T^*B}.
\]
Here $\hat{B}_{(-)}$ stands for the formal neighborhood of $B$ inside $(-)$.
\end{prop}

\begin{proof}
Since the formal neighborhood of the $0$-section in a commutative group scheme splits, \emph{i.e.}, it is isomorphic by the logarithm to the formal completion of the $0$-section in its normal bundle, we have
\[
\hat{B}_{P} = \hat{B}_{N_{B/P}}.
\]
So it suffices to show that $\kappa_\Theta$ (together with $\sigma$) identifies the normal bundle $N_{B/P}$ and the cotangent bundle $T^*B$ over $B$. This is achieved in two steps. First, the \'etale $B$-morphism $\kappa_\Theta$ induces an isomorphism $N_{B/P} \simeq N_{B/A}$. Then the restriction of the symplectic form $\sigma$ on $B$~to 
\[
A = M^{\mathrm{sm}} \subset M
\]
induces an isomorphism $N_{B/A} \simeq T^*B$. The composition of these two isomorphism gives the desired one.
\end{proof}

Now we may state our main conjecture precisely using Proposition \ref{prop1.4}.

\begin{conj}\label{main_conj2}
For an integrable system $\pi_M: (M, \sigma) \to B$, let $\Theta$ be any $\pi_M$-relatively ample line bundle (\ref{choice}). Under the isomorphism $\hat{\kappa}_\Theta$ between the formal neighborhoods of $B$ in~$P$ and $T^*B$, there exists an isomorphism
\begin{equation*}
\phi_{\mathrm{FM}}(\Omega_M^k)|_{\hat{B}_P} \simeq \mathrm{gr}(P_k)|_{\hat{B}_{T^*B}} \in \mathrm{Coh}(\hat{B}).
\end{equation*}
\end{conj}

\begin{rmk}
The statements of Conjecture \ref{main_conj2} come in pairs. On one hand, the symplectic form $\sigma$ induces an isomorphism $\Omega_M^k \simeq \Omega_M^{2n - k}$, where $2n = \dim M$. On the other hand, we have $\mathrm{gr}(P_k) \simeq \mathrm{gr}(P_{2n - k})$ by the relative Hard Lefschetz theorem.
\end{rmk}


\section{Singular supports}

\subsection{Overview}
In this section, we review the decomposition theorem of $\pi_M: M \to B$ and describe the reduced support
\[
\mathrm{supp}^{\mathrm{red}}( \mathrm{gr}(P_k) ) \subset T^*B
\]
following Migliorini--Shende \cite{disc} using the group scheme $A$. As this coincides with the singular support of the underlying perverse sheaf, in this section we can view $P_k$ merely as a perverse sheaf; in particular the filtration $F_\bullet P_k$ provided by Saito's theory does not play a role.

We conclude this section by introducing $\Lambda' \subset P$ which serves as the counter-part at the Fourier--Mukai side of the singular support associated with the decomposition theorem. A comparison between $\Lambda$ and $\Lambda'$ is deduced in Corollary \ref{cor2.4}.

\subsection{The decomposition theorem}
As before, we assume that 
\[
\dim M = 2 \dim B = 2n.
\]
Since a Lagrangian fibration has equidimesional fibers of dimension $n$ \cite{Equidim}, the decomposition theorem \cite{BBD} yields
\[
R\pi_{M*} \BQ_M[n] \simeq \bigoplus_{i=0}^{2n} P_k [-k], \quad P_k= {^\mathfrak{p}\CH}^k( R\pi_* \BQ_M[n] ).
\]
Here each perverse sheaf $P_k$ is semisimple.

We note that in the cases of Example \ref{ex1.1} Ng\^o's support theorem \cite{Ngo} implies that all $P_k$ have full support. In particular, each $P_k$ is given by the intermediate extension of the local systems obtained from the smooth locus of the integrable system $\pi_M: M \to B$. More refined information regarding the topology of $\pi_M$ is encoded in the \emph{singular supports} which are certain conical Lagrangians subvarieties of $T^*B$. We describe them in the following section for integrable systems.

\subsection{Singular supports}
For any proper morphism $f: X\to Y$ between nonsingular varieties, Migliorini and Shende \cite{disc} provide a concrete description of the singular supports
\[
\mathrm{SS}(Rf_*\BQ_X) \subset T^*Y
\]
associated with the decomposition theorem of $f$ using higher discriminants. 

Recall that for each $i \geq 0$, the higher discriminant $\Delta^i(f)$ is formed by points $y \in Y$ such that no $(i-1)$-dimensional subspace of the tangent space $T_yY$ at $y$ is transverse to $f$. We obtain a stratification
\[
Y= \Delta^0(f) \supset \Delta^1(f) \supset \Delta^2(f) \supset \dots,
\]
where each $\Delta^i(f) \subset Y$ is closed with
\begin{equation}\label{codim}
\mathrm{codim}_Y(\Delta^i(f)) \geq i.
\end{equation}
The closed subset $\Delta^i(f)$ generalizes the discriminant $\Delta^1(f)$ --- the locus where the fiber is singular. 

The following is the main result of \cite{disc}.

\begin{thm}[Migliorini--Shende]\label{MS2.1}
The singular support of $Rf_*\BQ_X$ is contained in the union of the conormal varieties to $i$-codimensional components of $\Delta^i(f)$ for all $i$.
\end{thm}

In the case of an integrable system $\pi_M: M \to B$, the higher discriminants $\Delta^i(\pi_M)$ are more concretely given by the $\delta$-stratification associated with the group scheme $A$ \cite[Proposition~4.3]{disc}.

\begin{prop}\label{prop2.2}
Let $A$ be the $B$-group scheme associated with an integrable system $\pi_M: M \to B$. We have
\[
\Delta^i(\pi_M) = \{b\in B\,|\,\delta(b) \geq i\}.
\]
\end{prop}

Using Proposition \ref{prop2.2}, we may reformulate Ng\^o's $\delta$-regularity in Proposition \ref{prop1.3} as
\[
\mathrm{codim}_B \Delta^i(\pi_M) \geq i,
\]
which is an immediate consequence of (\ref{codim}).

\subsection{The subvarieties $\Lambda \subset T^*B$ and $\Lambda' \subset P$}\label{Sec2.4}

Recall from (\ref{Lambda}) the conical Lagrangian $\Lambda \subset T^*B$; the ``support part'' of Theorem \ref{gr(P)} follows from Theorem \ref{MS2.1}:
\[
\mathrm{supp}^{\mathrm{red}}(\mathrm{gr}(P_k)) = \mathrm{SS}(P_k) \subset  \mathrm{SS}(R{\pi_M}_* \BQ_M)  \subset \Lambda.
\]

Furthermore, in view of Proposition \ref{prop2.2} the higher discriminants $\Delta^i(\pi_M)$ and the conical Lagrangian $\Lambda \subset T^*B$ are governed by the $\delta$-stratification on the base $B$. In the following, we construct a closed subvariety $\Lambda'\subset P$ using also $\Delta^i(\pi_M)$, which serves as
the counter-part of $\Lambda \subset T^*B$ for the Fourier--Mukai side.

For notational convenience, we denote by $G$ a commutative group space over $B$ which is either $A$ or $P$. Since $\kappa_\Theta$ preserves the affine parts, the $\delta$-functions calculating the dimensions of the maximal affine subgroups for $A$ and $P$ coincide. 

Let 
\[
1 \to R_b \to G_b \to H_b \to 1
\]
be the Chevalley decomposition for the group space $G$ over a closed point $b\in B$. We define the algebraic closed subset $\Delta_G^{\mathrm{aff}} \subset G$ to be the locus $g\in G$ such that $g$ lies in the affine part~$R_{\pi_M(g)}$ over $\pi_M(g) \in B$. 

The lemma below follows from Proposition \ref{prop1.3}.

\begin{lem}
Any irreducible component of $\Delta_G^{\mathrm{aff}}$ has dimension $\leq n$.
\end{lem}

We denote by
\[
\Lambda'_G \subset \Delta_G^{\mathrm{aff}}
\]
the union of all $n$-dimensional irreducible components. 

We now give a more concrete description of $\Lambda'_G$ in terms of $\Delta^i(\pi_M)$ parallel to that of $\Lambda$ in (\ref{Lambda}). For a locally closed subset $Z \subset B$, there exists an open dense $V\subset Z$ such that the Chevalley decomposition over $V$ is of the form
\[
1\to R_V \to G|_V \to H_V \to 1
\]
with $R_V$ affine and $H_V$ abelian; \emph{c.f.}~the paragraph before \cite[Section 2.3]{dCRS}. We use $\overline{R_Z}$ to denote the Zariski closure of $R_V$ in $G$ for any choice of $V$ as above.


The following proposition can be compared with the definition (\ref{Lambda}) of $\Lambda \subset T^*B$.

\begin{prop}\label{prop2.3}
We have
\[
\Lambda'_G =  \bigcup_i \bigcup_{Z_i} \overline{R_{Z_i}} \subset G.
\]
Here $Z_i$ runs through purely $i$-codimensional irreducible components of $\Delta^i(\pi_M)$.
\end{prop}

\begin{proof}
Since specializations preserve affine parts of the groups, we have
\[
 \Lambda'_G \supseteq \bigcup_i \bigcup_{Z_i} \overline{R_{Z_i}}.
\]
To prove the other inclusion, we take $W$ to be an irreducible component of $\Lambda_G'$; by definition it is purely of dimension $n$. We consider $Z = \pi_M(W) \subset B$. Assume that a general point $b \in Z$ satisfies $\delta(b) = i$. Hence a general fiber of $W \to Z$ has dimension $i$. This implies that $Z$ has codimension $i$ since
\[
\dim W = n. 
\]
On the other hand, we have $Z \subset \Delta^i(\pi_M)$ where the latter has codimension at least $i$ by~(\ref{codim}). Therefore $Z$ is an $i$-dimensional irreducible component of $\Delta^i(\pi_M)$, and the irreducible component $W$ is of the form $\overline{R_Z}$. This proves that $W \subset \Lambda_G'$.
\end{proof}

In the following corollary, we relate $\Lambda \subset T^*B$ to
\[
\Lambda'_A \subset A, \quad \Lambda'_P \subset P.
\]

\begin{cor}\label{cor2.4}
\begin{enumerate}
    \item[(i)] For the identification $\hat{B}_A = \hat{B}_{T^*B}$ of formal neighborhoods induced by the symplectic form $\sigma$, we have
    \[
    \Lambda'_A|_{\hat{B}_A} = \Lambda|_{\hat{B}_{T^*B}}.
    \]
    \item[(ii)] For any choice of $\Theta$ in (\ref{choice})  which identifies $\hat{B}_P$ and $\hat{B}_{T^*B}$ as in Proposition \ref{prop1.4}, we have
\[
  \Lambda_P'|_{\hat{B}_P} =   \Lambda|_{\hat{B}_{T^*B}}.
\]
\end{enumerate}
\end{cor}

\begin{proof}
The first part follows from Proposition \ref{prop2.3} and the proof of \cite[Proposition 4.3]{disc}, that 
\[
\overline{R_{Z_i}}|_{\hat{B}_A}  = \overline{T^*_{Z_i}B}|_{\hat{B}_{T^*B}}.
\]

The second part follows from the fact that the \'etale map between the groups
\[
{\kappa}_\Theta: A \to P
\]
preserves the affine parts.
\end{proof}

From now on we set
\[
\Lambda' : = \Lambda'_P \subset P.
\]

\section{Fourier--Mukai I: the perverse--Hodge symmetry}

\subsection{Overview}
We review Arinkin--Fedorov's Fourier--Mukai transform and deduce some of its basic properties. This recovers the border cases $k=0,2n$ of Conjecture \ref{main_conj2}; see Corollary~\ref{cor3.2}. Then we discuss a reformulation from the Fourier--Mukai viewpoint of Matsushita's theorem \cite{Ma3} concerning the higher direct image of $\CO_M$ and its generalization --- the perverse--Hodge symmetry (Conjecture \ref{conj3.4}) \cite{SY,SY2}. We conclude this section by proving Proposition~\ref{prop0.6}.

\subsection{Fourier--Mukai functors} \label{Sec3.2}
We start with a brief review of the classical Fourier--Mukai transform \cite{Mukai}.

Let $A$ be an $n$-dimensional abelian variety with $P = \mathrm{Pic}^0(A)$ its dual. Then a universal family induces a canonical (normalized) Poincar\'e line bundle $\CP$ on $A \times P$; it further induces two functors for the bounded derived categories of coherent sheaves:
\[
\Phi_{\mathrm{FM}}: D^b\mathrm{Coh}(P) \to D^b\mathrm{Coh}(A), \quad \CE \mapsto R{q_{A*}}(q_P^* \CE \otimes \CP) 
\]
and
\[
\phi_{\mathrm{FM}}:  D^b\mathrm{Coh}(A) \to D^b\mathrm{Coh}(P),\quad \CE \mapsto R{q_{P*}}(q_A^* \CE \otimes \CP^\vee )[n].
\]
Here $q_A$ and $q_P$ are the natural projections from $A\times P$ to the corresponding factors. Both functors $\Phi_{\mathrm{FM}}$ and $\phi_{\mathrm{FM}}$ are equivalences of categories, and they are inverses of each other:
\[
\phi_{\mathrm{FM}} \circ \Phi_{\mathrm{FM}}  \simeq \mathrm{id}_{D^b\mathrm{Coh}(P)}, \quad \Phi_{\mathrm{FM}} \circ \phi_{\mathrm{FM}}  \simeq \mathrm{id}_{D^b\mathrm{Coh}(A)}.
\]

Arinkin--Fedorov \cite{AF} generalizes the picture above and their construction works for a large class of degenerate abelian schemes including integrable systems; see also \cite{A1, A2} for the case of compactified Jacobians. Here we focus on the case of the integrable system $\pi_M: M \to B$ with 
\[
\dim M = 2 \dim B = 2n.
\]
Recall the relative Picard space $P$ which is smooth over $B$, and the normalized universal line bundle $\CL$ over $M \times_B P$. Similarly, we have two functors
\[
\Phi_{\mathrm{FM}}: D\mathrm{QCoh}(P) \to D\mathrm{QCoh}(M), \quad \CE \mapsto R{q_{M*}}(q_P^* \CE \otimes \CL) 
\]
and
\[
\phi_{\mathrm{FM}}:  D\mathrm{QCoh}(M) \to D\mathrm{QCoh}(P),\quad \CE \mapsto R{q_{P*}}(q_M^* \CE \otimes \CL^\vee ) \otimes \pi_P^* \omega_B^\vee[n].
\]
When $\pi_M$ has singular fibers, the above two functors are no longer equivalences. Nevertheless, we still have by \cite{AF} that  
\begin{equation} \label{eq:phiPhi}
\phi_{\mathrm{FM}} \circ \Phi_{\mathrm{FM}} \simeq \mathrm{id}_{D\mathrm{QCoh}(P)}: D\mathrm{QCoh}(P) \to D\mathrm{QCoh}(P),
\end{equation}
and in particular $\Phi_{\mathrm{FM}}$ is fully-faithful. We note that since $q_M: M\times_BP \to P$ is proper, the functor $\phi_{\mathrm{FM}}$ preserves bounded coherent complexes.

\begin{prop}\label{prop3.1}
For $\CK \in D^b\mathrm{Coh}(B)$, we have
\begin{equation*}\label{eqn3.1}
\Phi_{\mathrm{FM}}(0_{B*}\CK) \simeq \pi_M^*\CK, \quad \phi_{\mathrm{FM}}(\pi_M^*\CK) \simeq 0_{B*} \CK.
\end{equation*}
\end{prop}

\begin{proof}
By \eqref{eq:phiPhi} it suffices to prove the first statement. Since $\CL$ is trivialized along the $0$-section~$0_B:B \rightarrow P$, we have
\begin{equation}\label{3.4_1}
0_B^* \CL \simeq 0_B^* \CL^\vee \simeq \CO_M.
\end{equation}
We use $0_M: M \rightarrow M\times_BP$ to denote the base change of the $0$-section $0_B: B \rightarrow P$. We have
\begin{align*}
\Phi_{\mathrm{FM}}(0_{B*}\CK) & \simeq Rq_{M*}(q_P^*0_{B*}\CK \otimes \CL) \\
& \simeq Rq_{M*}(0_{M*}\pi_M^*\CK \otimes \CL) \\
& \simeq Rq_{M*}0_{M*}(\pi_M^*\CK \otimes 0_M^*\CL) \\
& \simeq \pi_M^*\CK.
\end{align*}
Here the second isomorphism is the base change $q_P^*0_{B*} \simeq 0_{M*}\pi_M^*$, the third isomorphism is the projection formula, and the fourth isomorphism follows from $q_M \circ 0_M = \mathrm{id}_M$ and \eqref{3.4_1}.
\end{proof}

Applying Proposition \ref{prop3.1} to $\CK = \CO_B$, we can verify the border cases of Conjecture \ref{main_conj2}.

\begin{cor}\label{cor3.2}
Conjecture \ref{main_conj2} holds for $k = 0, 2n$.
\end{cor}

\begin{proof}
We have
\[
\phi_{\mathrm{FM}}(\CO_M) \simeq 0_{B*}\CO_B \in \mathrm{Coh}(P), \quad \mathrm{gr}(P_0) \simeq 0_{B*}\CO_B \in \mathrm{Coh}(T^*B).
\]
Both are structure sheaves of the $0$-sections; in particular, they are isomorphic in the formal neighborhoods of $B$.
\end{proof}

As in the abelian variety case, there is a (derived) Pontryagin product on $P$. Consider the addition map
\[
m_P : P \times_B P \to P.
\]
We define
\[
\star^R : D\mathrm{QCoh}(P) \times D\mathrm{QCoh}(P) \to D\mathrm{QCoh}(P), \quad (\CE_1, \CE_2) \mapsto Rm_{P*}(q_1^*\CE_1 \otimes^L q_2^*\CE_2)
\]
where $q_1$ and $q_2$ are the two projections from $P \times_B P$.

\begin{prop} \label{prop3.3}
For $\CE_1, \CE_2 \in D\mathrm{QCoh}(P)$ and $\CK_1, \CK_2 \in D^b\mathrm{Coh}(M)$, we have
\[
\Phi_{\mathrm{FM}}(\CE_1 \star^R \CE_2) \simeq \Phi_{\mathrm{FM}}(\CE_1) \otimes^L \Phi_{\mathrm{FM}}(\CE_2), \quad \phi_{\mathrm{FM}}(\CK_1 \otimes^L \CK_2) \simeq \phi_{\mathrm{FM}}(\CK_1) \star^R \phi_{\mathrm{FM}}(\CK_2).
\]
\end{prop}

\begin{proof}
The two statements are similar; we only prove the first. By the theorem of the cube applied to $M \times_B P \times_B P$, we have
\begin{equation} \label{eq:cube}
(\mathrm{id} \times m_P)^*\CL \simeq q_{12}^*\CL \otimes q_{13}^*\CL
\end{equation}
where the $q_{ij}$ are the natural projections from $M \times_B P \times_B P$ to the respective factors. Then
\begin{align*}
\Phi_{\mathrm{FM}}(\CE_1 \star^R \CE_2) & \simeq Rq_{M*}(q_P^*Rm_{P*}(q_1^*\CE_1 \otimes^L q_2^*\CE_2) \otimes \CL) \\
& \simeq Rq_{M*} (R(\mathrm{id} \times m_P)_*q_{23}^*(q_1^*\CE_1 \otimes^L q_2^*\CE_2) \otimes \CL) \\
& \simeq Rq_{M*} R(\mathrm{id} \times m_P)_*(q_{12}^*q_P^*\CE_1 \otimes^L q_{13}^*q_P^*\CE_2 \otimes (\mathrm{id} \times m_P)^*\CL) \\
& \simeq Rq_{M*} Rq_{13*}(q_{12}^*q_P^*\CE_1 \otimes^L q_{13}^*q_P^*\CE_2 \otimes q_{12}^*\CL \otimes q_{13}^*\CL) \\
& \simeq Rq_{M*} (Rq_{13*}q_{12}^*(q_P^*\CE_1 \otimes \CL) \otimes^L (q_P^*\CE_2 \otimes \CL)) \\
& \simeq Rq_{M*} (q_M^*Rq_{M*}(q_P^*\CE_1 \otimes \CL) \otimes^L (q_P^*\CE_2 \otimes \CL)) \\
& \simeq \Phi_{\mathrm{FM}}(\CE_1) \otimes^L \Phi_{\mathrm{FM}}(\CE_2).
\end{align*}
Here the second isomorphism is the base change $q_P^*Rm_{P*} \simeq R(\mathrm{id} \times m_P)_*q_{23}^*$, the third isomorphism is the projection formula, the fourth isomorphism follows from $q_M \circ (\mathrm{id} \times m_P) = q_M \circ q_{13}$ and \eqref{eq:cube}, the fifth isomorphism is again the projection formula, the sixth isomorphism is the base change $Rq_{13*}q_{12}^* \simeq q_M^*Rq_{M*}$, and the last isomorphism follows from the projection formula and the definition of $\Phi_{\mathrm{FM}}$.
\end{proof}

\subsection{Pushforward and Fourier--Mukai}
The Fourier--Mukai functors $\Phi_{\mathrm{FM}}$ and $\phi_{\mathrm{FM}}$ are also compatible with the pushforwards
\[
R\pi_{P*}: D\mathrm{QCoh}(P) \to D\mathrm{QCoh}(B), \quad R\pi_{M*}: D^b\mathrm{Coh}(M) \rightarrow D^b\mathrm{Coh}(B).
\]

\begin{prop}\label{prop3.2}
For $\CE \in D\mathrm{QCoh}(P)$ and $\CK \in D^b\mathrm{Coh}(M)$, we have
\[
Ls_B^*\,\Phi_{\mathrm{FM}} (\CE) \simeq R\pi_{P*}\CE, \quad L0_B^*\,\phi_{\mathrm{FM}} (\CK) \simeq R\pi_{M*}\CK \otimes \omega_B^\vee[n].
\]
\end{prop}

\begin{proof}
The two statements are similar; we only prove the second. We have
\begin{align*}
    L0_B^*\,\phi_{\mathrm{FM}}(\CK) & \simeq L0_B^* Rq_{P*} ( q_{M}^*\CK \otimes \CL^\vee) \otimes \omega_B^\vee [n] \\
    & \simeq R\pi_{M*} L0_M^* (q_M^*\CK\otimes \CL^\vee ) \otimes \omega_B^\vee [n] \\
    & \simeq R\pi_{M*} ( \CK \otimes 0_M^* \CL^\vee) \otimes \omega_B^\vee [n] \\
    & \simeq R\pi_{M*} \CK \otimes \omega_B^\vee[n].
\end{align*}
Here the second isomorphism is the base change $L0_B^*Rq_{P*} \simeq R\pi_{M*}L0_M^*$, the third isomorphism is given by $q_M \circ 0_M  = \mathrm{id}_M$, and the fourth isomorphism follows from (\ref{3.4_1}).
\end{proof}


The following example explains a connection between the Fourier--Mukai transform and a theorem of Matsushita \cite{Ma3}. It can be thought of as a first example where the decomposition theorem and the Fourier--Mukai transform are related.

\begin{example}\label{ex3.3}
Applying Proposition \ref{prop3.2} to $\CK = \CO_M$, we obtain from Proposition \ref{prop3.1} that
\[
R\pi_{M*}\CO_M \simeq L0_B^*\CO_B  \otimes \omega_B[-n].
\]
The object $L0_B^*\CO_B$ is the derived self-tensor of the structure sheaf $\CO_B$ of the $0$-section $B \subset P$. Using either the Koszul resolution or the derived self-intersection formula \cite{AC}, we have
\begin{equation} \label{eq:self}
L0_B^* \CO_B \simeq \bigoplus_{k=0}^n \wedge^k N^\vee_{B/P}[k] \simeq \bigoplus_{k=0}^n T_B^k[k].
\end{equation}
Therefore
\begin{equation}\label{Mat0}
R\pi_{M*} \CO_M \simeq \bigoplus_{k=0}^nT_B^k \otimes \omega_B[k-n] \simeq \bigoplus_{k=0}^n \Omega_B^k[-k].
\end{equation}
The equation (\ref{Mat0}) implies simultaneously that
\begin{enumerate}
    \item[(i)] the derived direct image admits a splitting into sheaves
    \begin{equation} \label{eq:kollar}
    R\pi_{M*} \CO_M \simeq \bigoplus_{k=0}^n R^k\pi_{M*} \CO_M[-k],
    \end{equation}
    and
    \item[(ii)] each higher direct image is given by
    \[
    R^k\pi_{M*} \CO_M \simeq \Omega_B^k.
    \]
\end{enumerate}

In fact, both statements are known to hold for any Lagrangian fibration $\pi_M: M \to B$ beyond the case of integrable systems we consider here. The statement (i) was deduced from Saito's enhancement of the decomposition theorem by passing to graded pieces.  The statement (ii) is a theorem of Matsushita \cite{Ma3} which was proven for any Lagrangian fibration; Matsushita's argument relies on the decomposition theorem (i) as well as Hodge-theoretic and birational-geometric techniques.

For an integrable system $\pi_M: M\to B$, the Fourier--Mukai transform provides an alternative argument, yielding both (i) and (ii) from a single calculation 
\[
Rq_{P*}\CL \simeq 0_{B*}\omega_B[-n], 
\]
which is proven in \cite{AF} essentially from Ng\^o's $\delta$-inequality (Proposition \ref{prop1.3} in \cite{AF}). In particular, the Fourier--Mukai approach suggests that (i) and (ii) are both determined by the topology of the integrable system. \qed
\end{example}

\subsection{The perverse--Hodge symmetry}

As illustrated in the example above, the object $R\pi_{M*} \CO_M \in D^b\mathrm{Coh}(B)$ is governed by the Fourier--Mukai identity
\begin{equation}\label{FM3.4}
\phi_{\mathrm{FM}}(\CO_M) \simeq 0_{B*}\CO_B.
\end{equation}

In this section we discuss a generalization (Conjecture \ref{conj3.4} below) of Example \ref{ex3.3} concerning the objects
\[
R\pi_{M*} \Omega_M^k \in D^b\mathrm{Coh}(B), \quad k=0,1, \cdots, 2n,
\]
where Conjecture \ref{main_conj2} plays the role of (\ref{FM3.4}). This is a version of the proposal \cite{SY2} for the sheaf-theoretic perverse--Hodge symmetry which simultaneously generalizes Matsushita's theorem~(\ref{Mat0}) and categorifies the cohomological ``perverse = Hodge'' identity (\ref{P=F}).

For our purpose here, we view $P_k$ of the decomposition theorem a Hodge module on the base $B$; in particular, we consider its graded object 
\[
\mathrm{gr}(P_k) \in \mathrm{Coh}(T^*B).
\]

\begin{conj}[Categorified perverse--Hodge symmetry]\label{conj3.4}
We have
\[
R\pi_{M*} \Omega_M^k \simeq  ( L0_B^*\, \mathrm{gr}(P_k) )^\vee \in D^b\mathrm{Coh}(B).
\]
\end{conj}

\begin{rmk}\label{remark}
\begin{enumerate}
\item[(i)]Conjecture \ref{conj3.4} is expected to hold for not only integrable systems, but any Lagrangian fibration $\pi_M: M \to B$ with $M$ and $B$ nonsingular. As we show in the proof of Proposition \ref{prop0.6} below, Conjecture \ref{conj3.4} follows from Conjecture \ref{P=F2}. In particular, by \cite[Theorem 0.4]{SY2} we have verified Conjecture \ref{conj3.4} for Lagrangian fibrations induced by the Hilbert scheme of points on a surface that admits an elliptic fibration.

\item[(ii)] For an integrable system, Conjecture \ref{conj3.4} can be rewritten as
\[
L0_{B}^*\,\phi_{\mathrm{FM}}(\Omega^k_B) \simeq L0_B^*\, \mathrm{gr}(P_k)
\]
where we applied Proposition \ref{prop3.2} and the fact that the vector bundle $\Omega_M^k$ is self-dual. Consequently, Conjecture \ref{main_conj2} implies Conjecture \ref{conj3.4}.
\item[(iii)] When $B$ is projective, we recover the identity (\ref{P=F}) from Conjecture \ref{conj3.4} by taking cohomology $H^i(B, -)$. This follows from applying Laumon's formula (\emph{c.f.}~\cite[Theorem~2.4]{PS}) to the projective map $B \to \mathrm{pt}$.
\end{enumerate}
\end{rmk}

The following diagram illustrates the role of the Fourier--Mukai transform:
\begin{equation*}
    \begin{tikzcd}
    \textup{Conj.~\ref{main_conj2}: FM/Decomp.}  \arrow[d, snake it, "\textup{Specialize:~} k=0"] \arrow[rr, dashed, "L0_B^*(-)"] & & \textup{Conj.~\ref{conj3.4}: Categorified Perv.~= Hodge} \arrow[d, snake it, "\textup{Specialize:~} k=0"] \\
      \textup{FM (\ref{FM3.4}) for } \CO_M  \arrow[rr,dashed, "L0_B^*(-)"] & & \textup{Matsushita:~} R^i\pi_{M*}\CO_M\simeq \Omega_B^i.
\end{tikzcd}
\end{equation*}

Next, we complete the proof of Proposition \ref{prop0.6}.

\begin{proof}[Proof of Proposition \ref{prop0.6}]

We relate both sides of Conjecture \ref{conj3.4} to the perverse--Hodge complexes
\[
\CG_{k,i} : = \mathrm{gr}_{-i}\mathrm{DR}(P_{k})[n-k] \in D^b\mathrm{Coh}(B)
\]
of \cite{SY2} and reduce Conjecture \ref{conj3.4} to the symmetry (\ref{P=F2}):
\[
\CG_{k,i} \simeq \CG_{i,k}.
\]

Recall that if we view $P_k$ as a Hodge module on $B$, it carries the structure as a $\CD_B$-module endowed with a good filtration $F_\bullet P_k$. The $i$-th graded piece gives a sheaf of $\CO_B$-module $\mathrm{gr}_i P_k$. We consider the de Rham complex 
\[
\mathrm{DR}(P_k) = [P_k \to P_k\otimes \Omega_B^1 \to \cdots \to P_k \otimes \Omega_B^n][n];
\]
with the induced filtration
\[
F_i\mathrm{DR}(P_k) = [F_iP_k \to F_{i+1}P_k\otimes \Omega_B^1 \to \cdots \to F_{i+n}P_k\otimes \Omega_B^n][n].
\]
The associated graded object of the de Rham complex induces objects 
\begin{equation}\label{dR1}
\mathrm{gr}_i\mathrm{DR}(P_k) =  [\mathrm{gr}_iP_k \to \mathrm{gr}_{i+1}P_k\otimes \Omega_B^1 \to \cdots \to \mathrm{gr}_{i+n}P_k\otimes \Omega_B^n][n]
\end{equation}
taking values in the bounded derived category of coherent sheaves on $B$. In particular, the perverse--Hodge complex $\CG_{k,i}$ encodes the information of the $k$-th and the $i$-th graded pieces of the perverse and the Hodge filtrations respectively, and thus (\ref{P=F2}) is a perverse--Hodge symmetry.

To relate the perverse--Hodge complexes to the more classical object of the left-hand side of Conjecture~\ref{conj3.4}, we evoke here Saito's formula \cite[2.3.7]{S1} (see also \cite[Section 2.2]{SY2}):
\[
R\pi_{M*}\mathrm{gr}_{-k}\mathrm{DR}(\BQ_M[2n]) \simeq \mathrm{gr}_{-k} \mathrm{DR}(R\pi_{M*} \BQ_M[2n]) \simeq \bigoplus_{i} \mathrm{gr}_{-k} \mathrm{DR}(P_i)[n-i]. 
\]
Here $\BQ_M[2n]$ is the Hodge module associated with the perverse sheaf $\BQ_M[2n]$ given by the trivial local system, whose de Rham complex has graded piece 
\[
\mathrm{gr}_{-k}^F \mathrm{DR}(\BQ_M[2n]) = \Omega_M^k[2n-k].
\]
Therefore, we obtain that
\begin{equation*}
R\pi_{M*} \Omega_M^k \simeq \bigoplus_{i} \CG_{i, k} [k-2n];
\end{equation*}
in particular, by the Grothendieck--Verdier duality, 
\begin{equation}\label{LHS}
(R\pi_{M*} \Omega_M^k )^\vee \simeq \left(\bigoplus_{i} \CG_{i,k} \right) \otimes \pi_M^*\omega_B^\vee [k-n].
\end{equation}

On the other hand, we have
\begin{equation}\label{RHS1}
L0_B^*\, \mathrm{gr}(P_k) \simeq Rp_*( \mathrm{gr}(P_k) \otimes^L_{\CO_{T^*B}} \CO_B )
\end{equation}
with $p: T^*B \to B$ the natural projection. By taking the Koszul resolution of the structure sheaf $\CO_B$ of the $0$-section $B \subset T^*B$, we see from the expression (\ref{dR1}) that the right-hand side of (\ref{RHS1}) is
\[
Rp_*( \mathrm{gr}(P_k) \otimes^L_{\CO_{T^*B}} \CO_B ) \simeq \left(\bigoplus_{i} \CG_{k,i}\right) \otimes \pi_M^* \omega_B^\vee[k-n];
\]
it is matched with (\ref{LHS}) through the perverse--Hodge symmetry $\CG_{k,i} \simeq \CG_{i,k}$.

So far we have proved that (\ref{P=F2}) implies Conjecture \ref{conj3.4}. The first part of Proposition \ref{prop0.6} then follows from Remark \ref{remark} (ii), and the second part follows from Remark \ref{remark} (iii).
\end{proof}

\section{Fourier--Mukai II: Lagrangian Cohen--Macaulay sheaves}

\subsection{Overview}
In this section, we prove Theorem \ref{thm0.3} which shows that the object
\[
\phi_{\mathrm{FM}}( \Omega_M^k )|_{\hat{B}_P} 
\]
is a Cohen--Macaulay sheaf supported on the conical Lagrangian $\Lambda'|_{\hat{B}_P}$. This is matched with the singular support of $P_k$ in view of Theorem \ref{gr(P)} and Corollary \ref{cor2.4}.

Our method is built on the ideas of Arinkin \cite{A1} and Arinkin--Fedorov \cite{AF}.

\subsection{Cohen--Macaulay sheaves}

Let $\CF$ be a coherent sheaf on $P$. We say that $\CF$ is \emph{Cohen--Macaulay} if for any closed point $x \in P$, we have
\[
\mathrm{depth}_{\CO_{P,x}}(\CF_x) = \dim \mathrm{supp}(\CF_x).
\]
In particular, $\CF$ is Cohen--Macaulay of pure dimension $d$ if and only if its Verdier dual
\[
R\CH\mathrm{om}(\CF, \omega_P) \in D^b\mathrm{Coh}(P)
\]
is concentrated in degree $d$.

The following theorem is a Cohen--Macaulay criterion for the Fourier--Mukai transform of a locally free sheaf.

\begin{thm}\label{thm4.1}
Let $\CK$ be a locally free sheaf on $M$. Let $Z \subset P$ be a closed subset for which each irreducible component has dimension $\leq n$. If
\[
\mathrm{supp}^{\mathrm{red}}(\phi_{\mathrm{FM}}(\CK)) \subset Z,
\]
then $\phi_{\mathrm{FM}}(\CK)$ is a Cohen--Macaulay sheaf of pure dimension $n$ concentrated in degree 0. 
\end{thm}

\begin{proof}
By definition the object 
\[
\phi_{\mathrm{FM}}(\CK) = R{q_{P*}}(q_M^* \CK \otimes \CL^\vee) \otimes \pi_P^*\omega_B^\vee [n] \in D^b\mathrm{Coh}(P)
\]
is concentrated in degrees $[-n,0]$. Concerning its Verdier dual, we have
\begin{align*}
R\CH\mathrm{om}(\phi_{\mathrm{FM}}(\CK), \omega_P) & \simeq R\CH\mathrm{om}( Rq_{P*}(q_M^*\CK \otimes \CL^\vee) \otimes \pi_P^*\omega_B^\vee [n], \omega_P ) \\
& \simeq Rq_{P*} R\CH\mathrm{om}(q_M^*\CK \otimes \CL^\vee, \omega_{q_P})\\
& \simeq Rq_{P*} (q_M^*\CK^\vee \otimes \CL \otimes \omega_{q_P} ).
\end{align*}
Here the second isomorphism follows from Grothendieck--Verdier duality. Since $\CK$ is locally free on $M$, we have that $q_M^*\CK^\vee \otimes \CL \otimes \omega_{q_P}$ is a (locally free) sheaf on $M\times_BP$. Therefore the object 
\[
R\CH\mathrm{om}(\phi_{\mathrm{FM}}(\CK), \omega_P) \in D^b \mathrm{Coh}(P)
\]
is concentrated in degrees $[0,n]$.

In conclusion, we know that:
\begin{enumerate}
    \item[(i)] $\phi_{\mathrm{FM}}(\CK)$ is concentrated in $[-n,0]$,
    \item[(ii)] its Verdier dual $R\CH\mathrm{om}(\phi_{\mathrm{FM}}(\CK), \omega_P)$ is concentrated in $[0,n]$, and
    \item[(iii)] both objects have supports of dimension $\leq n$.
\end{enumerate}
The only way for this to happen is that $\phi_{\mathrm{FM}}(\CK)$ is a sheaf concentrated in degree 0 of pure dimension $n$, and its Verdier dual is concentrated in degree $n$ (see \cite[Lemma 7.7]{A2}); in other words, $\phi_{\mathrm{FM}}(\CK)$ is a Cohen--Macaulay sheaf of pure dimension $n$ (concentrated in degree 0).
\end{proof}

\subsection{Proof of Theorem \ref{thm0.3}}

\begin{prop}\label{prop4.2}
Assume that $\CK\in D^b\mathrm{Coh}(M)$ underlies an $A$-equivariant bounded complex on $M$. Then we have
\begin{enumerate}
    \item[(i)] the support of $\phi_{\mathrm{FM}}(\CK)$ has dimension $\leq n$, and
    \item[(ii)] its restriction to $\hat{B} \subset P$ is contained in $\Delta_P^{\mathrm{aff}}$; \emph{i.e.},
    \[
\mathrm{supp}^{\mathrm{red}}( \phi_{\mathrm{FM}}(\CK))|_{\hat{B}} \subset \Delta_P^{\mathrm{aff}}|_{\hat{B}}.
\]
\end{enumerate}
\end{prop}

\begin{proof}
The proof essentially follows from the discussion in \cite[Sections 4 and 5]{AF}; we present it here for the reader's convenience.

Let $l \in P$ be a line bundle on $M$ with 
\[
 l \in \mathrm{supp}^{\mathrm{red}}(\phi_{\mathrm{FM}}(\CK)) \subset P.
\]
We write $b: = \pi_P(l) \in B$, and denote by $A_b, M_b, P_b$ the restrictions of $A, M, P$ over the closed point $b$ respectively. The (derived) restriction of the $A$-equivariant object $\CK$ on $M$ gives an $A_b$-equivariant object $\CK_b \in D^b\mathrm{Coh}(M_b)$.\\

\noindent {\bf Step 1.} By base change, we know that 
\begin{equation}\label{Step1}
H^j(M_b, \CK_b \otimes l_b^\vee) \neq 0
\end{equation}
for some $j \in \BZ$. The restriction of the line bundle $l_b^\vee$ to $A_b \subset M_b$ corresponds to a $\BC^*$-torsor over $A_b$. By \cite[Proposition 5.6]{AF} this $\BC^*$-torsor, denoted by $\widetilde{A}_b$, is actually a commutative group scheme given by the extension
\begin{equation}\label{step1_1}
1 \to \BC^* \to \widetilde{A}_b \to A_b \to 1.
\end{equation}
It acts on the total space of $l_b^\vee$ such that the subtorus $\BC^* \subset \widetilde{A}$ acts tautologically by dilations. Since $\CK_b$ is $A_b$-equivariant, this gives rise to a natural $\widetilde{A}_b$-action on the vector space (\ref{Step1}).\\

\noindent{\bf Step 2.} Since $\widetilde{A}_b$ is commutative, we may find a 1-dimensional sub-representation
\[
V \subset H^j(M_b, \CK_b \otimes l_b^\vee),
\]
which induces a character 
\[
\chi: \widetilde{A}_b \rightarrow \BC^*, \quad \chi|_{\BC^* \subset \widetilde{A}_b } = \mathrm{id}_{\BC^*}.
\]
It splits the extension (\ref{step1_1}). In particular, the restriction of $l_b^\vee$ to $A_b \subset M_b$ is a trivial line bundle; equivalently, the restriction of $l_b$ to $A_b$ is trivial.\\

\noindent{\bf Step 3.} Now we focus on the \emph{cohomology support locus}
\[
K_b = \{l \in \mathrm{Pic}^0(M_b)\,|\, l|_{A_b} \simeq \CO_{A_b}\} \subset \mathrm{Pic}^0(M_b);
\]
it is a quasi-subgroup in the sense of \cite[Definition 4.11]{AF}, whose identity component $K_b^0 \subset K_b$ is affine by \cite[Corollary 4.6]{AF}. In particular, we have 
\[
\dim K_b \leq \delta(b).
\]
By Steps 1 and 2, we know that 
\[
\mathrm{supp}^{\mathrm{red}}(\phi_{\mathrm{FM}}(\CK)) \subset \bigcup_{b\in B} K_b 
\]
where the latter $K:= \bigcup_{b\in B} K_b$ is a countable union of constructible sets satisfying:
\[
\dim K  \leq \dim \Delta_P^{\mathrm{aff}} \leq n, \quad
    K|_{\hat{B}} \subset \Delta_P^{\mathrm{aff}}|_{\hat{B}}.
\]
The proposition is concluded.
\end{proof}

We complete in the following the proof of Theorem \ref{thm0.3}. Note that in general the locally free sheaves $\Omega_M^k$ are not $A$-equivariant as illustrated in the example below.

\begin{example}
Let $\pi: M \to B$ be a smooth  elliptic fibration with a section, and let $E \subset M$ be a fiber. In this case the group scheme $A$ is identical to $M$ itself. If the tangent bundle tangent bundle $T_M$ is $A$-equivariant, then it is obtained as the pullback of a vector bundle from the base $B$; in particular, the short exact sequence
\[
0 \to T_E \to T_M|_E \to N_{M/S} \to 0
\]
splits. This implies that the boundary map of the long exact sequence 
\[
H^0(E, N_{E/M}) \to H^1(E, T_E)
\]
vanishes. On the other hand, this is exactly the Kodaira--Spencer map which is nontrivial when the fibration is not isotrivial.
\end{example}

Therefore, we are not allowed to apply Proposition \ref{prop4.2} to the locally free sheaf $\CK = \Omega_M^k$ directly. To solve this issue, we need the following lemma.

\begin{lem}\label{lem4.4}
Let
\[
0 \to \CE_1 \to \CE_2 \to \CE_3 \to 0
\]
be a short exact sequence in $\mathrm{Coh}(M)$. Then associated with the derived exterior power $\wedge^k\CE_2 \in D^b\mathrm{Coh}(M)$, we have finitely many objects in $D^b\mathrm{Coh}(M)$
\[
\CG_0, \cdots, \CG_i,\CG_{i+1}, \cdots, \CG_k = \wedge^k\CE_2
\]
with morphisms $\CG_{i-1} \to \CG_{i}$ fitting into the exact triangles 
\[
\CG_{i-1} \to \CG_i \to \wedge^{k-i} \CE_1 \otimes^L \wedge^i\CE_3 \xrightarrow{+1} \CG_{i-1}[1].
\]
\end{lem}
\begin{proof}
This follows from the proof of \cite[Corollary 2.2]{Lich} via taking compatible locally free resolutions of the short exact sequence $0\to \CE_1 \to \CE_2 \to \CE_3 \to 0$.
\end{proof}

\begin{proof}[Proof of Theorem \ref{thm0.3}]
We first note that it suffices to prove
\begin{equation}\label{main1}
\mathrm{supp}^{\mathrm{red}}(\phi_{\mathrm{FM}}(\Omega_M^k)) \subset K
\end{equation}
with $K = \bigcup_{b\in B}K_p \subset P$ given in the proof of Proposition \ref{prop4.2}. In fact, (\ref{main1}) implies that \[
\dim \mathrm{supp}^{\mathrm{red}}(\phi_{\mathrm{FM}}(\Omega_M^k)) \leq n.
\]
Therefore, by Theorem \ref{thm4.1} the object $\phi_{\mathrm{FM}}(\Omega_M^k)$ is a Cohen--Macaulay sheaf of dimension $n$ concentrated in degree 0. Furthermore, since
\[
\phi_{\mathrm{FM}}(\Omega_M^k)|_{\hat{B}} \subset K|_{\hat{B}} \subset \Delta^{\mathrm{aff}}_P|_{\hat{B}},
\]
the Cohen--Macaulay property ensures that the support is contributed by the purely codimension $i$ irreducible components of $\Delta^i(\pi_M)$ (see Proposition \ref{prop2.3}), \emph{i.e.}, 
\[
\phi_{\mathrm{FM}}(\Omega_M^k)|_{\hat{B}}  \subset \Lambda'|_{\hat{B}}.
\]

Now it remains to prove (\ref{main1}). We consider the short exact sequence
\begin{equation*}
0 \to \pi_M^* \Omega_B^1 \to \Omega_M^1 \to \Omega^1_{M/B} \to 0.
\end{equation*}
Here both the first and the third terms are $A$-equivariant coherent sheaves. In view of Lemma~\ref{lem4.4}, the vector bundle $\Omega_M^k = \wedge^k \Omega_M^1$ can be expressed in terms of an iterated extension via exact triangles of objects of the form
\begin{equation}\label{main2}
\pi_M^* \Omega_B^{k-i} \otimes \wedge^i \Omega^1_{M/B} \in D^b\mathrm{Coh}(M),
\end{equation}
where $\wedge^i(-)$ stands for the derived exterior power as in Lemma \ref{lem4.4}. Since each term (\ref{main2}) is $A$-equivariant, its Fourier--Mukai transform has reduced support lying in $K$ by Proposition~\ref{prop4.2}. Hence we obtain that the reduced support of the Fourier--Mukai transform of $\Omega_M^k$ also lies in~$K$. This complete the proof.
\end{proof}

\section{Smooth fibrations}\label{smooth}

\subsection{Overview}
In this section, we verify Conjecture \ref{main_conj2} for smooth Lagrangian fibrations which proves Theorem \ref{thm0.4}. Let $(M, \sigma)$ be a nonsingular holomorphic symplectic variety carrying a \emph{smooth} Lagrangian fibration $\pi_M : M \to B$ with a section $s_B: B \to M$. Let $\pi_P: P \to B$ be the relative Picard scheme. As before, we assume that 
\[
\dim M = 2 \dim B = 2n.
\]

\subsection{The $k = 1$ case}\label{Sec5.2}
We first treat the $k = 1$ case; our strategy is the following:
\begin{enumerate}
\item[(i)] We express both $\phi_{\mathrm{FM}}(\Omega^1_M)$ and $\mathrm{gr}(P_1)$ as an extension of sheaves that are (scheme-theoretically) supported on the $0$-sections of $P$ and $T^*B$. The individual terms are pairwise matched via the symplectic form $\sigma$ and the $\pi_M$-relatively ample line bundle~$\Theta$.

\item[(ii)] We match the extensions by means of the Gauss--Manin connection. An important property of smooth Lagrangian fibrations, known as the Donagi--Markman cubic condition (see Lemma \ref{dm}), is crucial in the matching.
\end{enumerate}

Consider the (co)tangent sequence for the smooth morphism $\pi_M : M \to B$. We have a commutative diagram
\begin{equation}\label{eq:cotan}
\begin{tikzcd}
0 \arrow{r}{} & T_{M/B} \arrow{r}{} \arrow{d}{\sigma}[swap]{\simeq} & T_M \arrow{r}{} \arrow{d}{\sigma}[swap]{\simeq} & \pi_M^*T_B \arrow{r}{} \arrow{d}{\sigma}[swap]{\simeq} & 0 \\
0 \arrow{r}{} & \pi_M^*\Omega^1_B \arrow{r}{} & \Omega^1_M \arrow{r}{} & \Omega^1_{M/B} \arrow{r}{} & 0
\end{tikzcd}
\end{equation}
where the vertical arrows are isomorphisms induced by the symplectic form $\sigma$. Since $\Omega^1_{M/B}$ is pulled back from $B$, the natural morphism
\[
\pi_M^*\pi_{M*}\Omega^1_{M/B} \to \Omega^1_{M/B}
\]
is an isomorphism. Then, substituting $\Omega^1_{M/B}$ by $\pi_M^*\pi_{M*}\Omega^1_{M/B}$ in \eqref{eq:cotan} and applying the Fourier--Mukai functor $\phi_{\mathrm{FM}}$ as in Proposition \ref{prop3.1}, we obtain a short exact sequence
\begin{equation} \label{eq:Omega1}
0 \to 0_{B*}\Omega^1_B \to \phi_{\mathrm{FM}}(\Omega^1_M) \to 0_{B*}\pi_{M*}\Omega^1_{M/B} \to 0
\end{equation}
in $\mathrm{Coh}(P)$.

On the other side, the Hodge module $P_1$ corresponds to the variation of Hodge structures on~$R^1\pi_{M*}\mathbb{Q}_M$. We set $\CV_1 = R^1\pi_{M*}\mathbb{Q}_M \otimes_\BQ \CO_B$ with grades pieces $\CV^{i, 1-i} = \mathrm{gr}^i\CV_1$. The associated graded~$\mathrm{gr}(P_1)$ may be viewed either as a coherent sheaf on~$T^*B$ whose reduced support is the $0$-section, or as a Higgs bundle on $B$. The Higgs bundle is an extension of~\mbox{$\CV^{1, 0} \simeq \pi_{M*}\Omega^1_{M/B}$} by $\CV^{0, 1} \simeq R^1\pi_{M*}\CO_M$ (both with trivial Higgs field); in terms of coherent sheaves this gives a short exact sequence
\begin{equation} \label{eq:P1}
0 \to 0_{B*}R^1\pi_{M*}\CO_M \to \mathrm{gr}(P_1) \to 0_{B*}\pi_{M*}\Omega^1_{M/B} \to 0
\end{equation}
in $\mathrm{Coh}(T^*B)$.

There is a duality between~$R^1\pi_{M*}\CO_M$ and~$\pi_{M*}\Omega^1_{M/B}$ depending on a $\pi_M$-relatively ample line bundle $\Theta$. This, together with the isomorphism $T_B \simeq \pi_{M*}\Omega^1_{M/B}$ induced by~$\sigma$, yields an isomorphism
\begin{equation} \label{eq:r1pi}
R^1\pi_{M*}\CO_M \simeq \Omega^1_B.
\end{equation}
Comparing the two sequences \eqref{eq:Omega1} and \eqref{eq:P1}, we find that both terms on the left (resp.~right) are (scheme-theoretically) supported on the $0$-sections and match each other. It then remains to match the extensions, which is key to the proof of the $k = 1$ case.

\subsection{Identifying extensions}\label{Sec5.3}
As the terms on the sides of \eqref{eq:Omega1} and \eqref{eq:P1} are supported on the $0$-sections of $P$ and $T^*B$, their extensions only concern the formal neighborhoods of $B$ in~$P$ and $T^*B$. In particular, the isomorphism $\hat{\kappa}_\Theta$ of Proposition \ref{prop1.4} together with \eqref{eq:r1pi} induces an isomorphism of the extension groups
\begin{equation} \label{eq:extgp}
\mathrm{Hom}_{P}(0_{B*}\pi_{M*}\Omega^1_{M/B}, 0_{B*}\Omega^1_B[1]) \simeq \mathrm{Hom}_{T^*B}(0_{B*}\pi_{M*}\Omega^1_{M/B}, 0_{B*}R^1\pi_{M*}\CO_M[1])
\end{equation}

The identification \eqref{eq:extgp} can be expressed in concrete terms. In fact, we have by adjunction
\begin{align}
\phantom{\simeq{}} & \mathrm{Hom}_P(0_{B*}\pi_{M*}\Omega^1_{M/B}, 0_{B*}\Omega^1_B[1]) \nonumber \\
\simeq{} & \mathrm{Hom}_B(L0_B^*0_{B*}\pi_{M*}\Omega^1_{M/B}, \Omega^1_B[1]) \nonumber \\
\simeq{} & \mathrm{Hom}_B\left(\bigoplus_{i = 0}^n\wedge^iN_{B/P}^{\vee} \otimes \pi_{M*}\Omega^1_{M/B}[i], \Omega^1_B[1]\right) \nonumber \\
\simeq{} & \mathrm{Hom}_B(\pi_{M*}\Omega^1_{M/B}, \Omega^1_B[1]) \oplus \mathrm{Hom}_B(N_{B/P}^{\vee}[1] \otimes \pi_{M*}\Omega^1_{M/B}, \Omega^1_B[1]) \nonumber \\
\simeq{} & \mathrm{Hom}_B(\pi_{M*}\Omega^1_{M/B}, \Omega^1_B[1]) \oplus \mathrm{Hom}_B(\pi_{M*}\Omega^1_{M/B}, N_{B/P} \otimes \Omega^1_B) \label{eq:split1}
\end{align}
and respectively
\begin{align}
\phantom{\simeq{}} & \mathrm{Hom}_{T^*B}(0_{B*}\pi_{M*}\Omega^1_{M/B}, 0_{B*}R^1\pi_{M*}\CO_M[1]) \nonumber \\
\simeq{} & \mathrm{Hom}_B(L0_B^*0_{B*}\pi_{M*}\Omega^1_{M/B}, R^1\pi_{M*}\CO_M[1]) \nonumber \\
\simeq{} & \mathrm{Hom}_B\left(\bigoplus_{i = 0}^n\wedge^iT_B \otimes \pi_{M*}\Omega^1_{M/B}[i], R^1\pi_{M*}\CO_M[1]\right) \nonumber \\
\simeq{} & \mathrm{Hom}_B(\pi_{M*}\Omega^1_{M/B}, R^1\pi_{M*}\CO_M[1]) \oplus \mathrm{Hom}_B(T_B \otimes \pi_{M*}\Omega^1_{M/B}[1], R^1\pi_{M*}\CO_M[1]) \nonumber \\
\simeq{} & \mathrm{Hom}_B(\pi_{M*}\Omega^1_{M/B}, R^1\pi_{M*}\CO_M[1]) \oplus \mathrm{Hom}_B(\pi_{M*}\Omega^1_{M/B}, \Omega^1_B \otimes R^1\pi_{M*}\CO_M). \label{eq:split2}
\end{align}
Here we have used the derived self-intersection \eqref{eq:self} for $B \hookrightarrow P$ and the Koszul resolution for~$B \hookrightarrow T^*B$. The symplectic form $\sigma$ and the $\pi_M$-relatively ample line bundle $\Theta$ induce an isomorphism of the normal bundles
\[N_{B/P} \simeq \Omega_B^1.\]
This, together with the isomorphism \eqref{eq:r1pi}, identifies the summands of \eqref{eq:split1} and \eqref{eq:split2}.

The decompositions \eqref{eq:split1} and \eqref{eq:split2} are a priori not canonical. Yet by the natural truncations~$\tau_{\geq j}$ of $L0_B^*\CO_B$, we have canonical short exact sequences
\begin{multline*}
0 \to \mathrm{Hom}_B(\pi_{M*}\Omega^1_{M/B}, \Omega^1_B[1]) \to \mathrm{Hom}_P(0_{B*}\pi_{M*}\Omega^1_{M/B}, 0_{B*}\Omega^1_B[1]) \\
\to \mathrm{Hom}_B(\pi_{M*}\Omega^1_{M/B}, N_{B/P} \otimes \Omega^1_B) \to 0
\end{multline*}
and
\begin{multline*}
0 \to \mathrm{Hom}_B(\pi_{M*}\Omega^1_{M/B}, R^1\pi_{M*}\CO_M[1]) \to \mathrm{Hom}_{T^*B}(0_{B*}\pi_{M*}\Omega^1_{M/B}, 0_{B*}R^1\pi_{M*}\CO_M[1]) \\
\to \mathrm{Hom}_B(\pi_{M*}\Omega^1_{M/B}, \Omega^1_B \otimes R^1\pi_{M*}\CO_M) \to 0.
\end{multline*}
Moreover, both sequences admit a canonical splitting by pushing forward via $\pi_P: P \to B$ and~$p: T^*B \to B$, \emph{i.e.},
\[
R\pi_{P*}: \mathrm{Hom}_P(0_{B*}\pi_{M*}\Omega^1_{M/B}, 0_{B*}\Omega^1_B[1]) \to \mathrm{Hom}_B(\pi_{M*}\Omega^1_{M/B}, \Omega^1_B[1]),
\]
and
\[
Rp_*: \mathrm{Hom}_{T^*B}(0_{B*}\pi_{M*}\Omega^1_{M/B}, 0_{B*}R^1\pi_{M*}\CO_M[1]) \to \mathrm{Hom}_B(\pi_{M*}\Omega^1_{M/B}, R^1\pi_{M*}\CO_M[1]).
\]
We conclude that \eqref{eq:split1} and \eqref{eq:split2} are canonical decompositions with pairwise identifiable summands compatible with \eqref{eq:extgp}.

Let $\epsilon_{\Omega^1_M}$ and $\epsilon_{P_1}$ denote the respective extension classes of \eqref{eq:Omega1} and \eqref{eq:P1}.

\begin{lem}
We have $R\pi_{P*}\epsilon_{\Omega^1_M} = 0$ and $Rp_*\epsilon_{P_1} = 0$.
\end{lem}

\begin{proof}
The vanishing of $Rp_*\epsilon_{P_1}$ is immediate: this amounts to forgetting the Higgs field of of the Higgs bundle, so that it becomes the trivial extension of $\pi_{M*}\Omega^1_{M/B}$ by $R^1\pi_{M*}\CO_M$. For~$R\pi_{P*}\epsilon_{\Omega^1_M}$, we consider the short exact sequence
\begin{equation} \label{eq:pipush}
0 \to \Omega^1_B \to \pi_{P*}\phi_{\mathrm{FM}}(\Omega_M^1) \to \pi_{M*}\Omega^1_{M/B} \to 0
\end{equation}
obtained by applying $R\pi_{P*}$ to \eqref{eq:Omega1}. By Propositions \ref{prop3.1} and \ref{prop3.2}, the sequence \eqref{eq:pipush} is isomorphic to
\begin{equation} \label{eq:zerosec}
0 \to \Omega^1_B \to s_B^*\Omega_M^1 \to \pi_{M*}\Omega_{M/B}^1 \to 0
\end{equation}
obtained by applying $Ls_B^*$ to the cotangent sequence \eqref{eq:cotan}. The sequence \eqref{eq:zerosec} splits because for an abelian variety $A$, the cotangent space $\Omega^1_{A, 0}$ is naturally identified with $H^0(A, \Omega^1_A)$.
\end{proof}

In view of \eqref{eq:split1} and \eqref{eq:split2}, the extension class $\epsilon_{\Omega^1_M}$ is uniquely determined by a morphism
\begin{equation} \label{eq:Omega1'}
\pi_{M*}\Omega^1_{M/B} \to N_{B/P} \otimes \Omega^1_B
\end{equation}
and $\epsilon_{P_1}$ by a morphism
\begin{equation} \label{eq:P1'}
\pi_{M*}\Omega^1_{M/B} \to \Omega^1_B \otimes R^1\pi_{M*}\CO_M.
\end{equation}

\begin{prop} \label{prop5.2}
Under the isomorphisms
\[
N_{B/P} \simeq \Omega_B^1, \quad R^1\pi_{M*}\CO_M \simeq \Omega^1_B
\]
induced by the symplectic form $\sigma$ and the $\pi_M$-relatively ample line bundle $\Theta$, the two morphisms~\eqref{eq:Omega1'} and \eqref{eq:P1'} coincide. In particular, Conjecture \ref{main_conj2} holds for $\pi_M: M \to B$ for~$k = 1$.
\end{prop}

\begin{proof}
By the construction of the Spencer/Koszul resolution, the morphism \eqref{eq:P1'} is precisely the associated graded of the Gauss--Manin connection
\begin{equation} \label{eq:gm}
\overline{\nabla}: \CV^{1, 0} \to \Omega^1_B \otimes \CV^{0, 1}.
\end{equation}

To see the relation of \eqref{eq:Omega1'} with the Gauss--Manin connection, we consider the following commutative diagram
\begin{equation*}
\begin{tikzcd}
\mathrm{Hom}_M(\pi_M^*\pi_{M*}\Omega^1_{M/B}, \pi_M^*\Omega_B^1[1]) \arrow{r}{\phi_{\mathrm{FM}}}[swap]{\simeq} \arrow{d}{R\pi_{M*}} & \mathrm{Hom}_P(0_{B*}\pi_{M*}\Omega^1_{M/B}, 0_{B*}\Omega_B^1[1]) \arrow{d}{L0_B^*}\\
\mathrm{Hom}_B(R\pi_{M*}\pi_M^*\pi_{M*}\Omega^1_{M/B}, R\pi_{M*}\pi_M^*\Omega^1_B[1]) \arrow{r}{\otimes \omega_B^\vee[n]}[swap]{\simeq} & \mathrm{Hom}_B(L0_B^*0_{B*}\pi_{M*}\Omega^1_{M/B}, L0_B^*0_{B*}\Omega_B^1[1]).
\end{tikzcd}
\end{equation*}
Here the diagram commutes by Proposition \ref{prop3.2}, and the top row is an isomorphism since $\phi_{\mathrm{FM}}$ is fully-faithful (in fact, an equivalence). Applying the decomposition \eqref{eq:kollar} and the derived self-intersection \eqref{eq:self} to the bottom terms, we find
\begin{multline}\label{eq:dec1}
\mathrm{Hom}_B(R\pi_{M*}\pi_M^*\pi_{M*}\Omega^1_{M/B}, R\pi_{M*}\pi_M^*\Omega^1_B[1]) \\
\simeq{} \mathrm{Hom}_B\left(\bigoplus_{i = 0}^nR^i\pi_{M*}\CO_M \otimes \pi_{M*}\Omega^1_{M/B}[-i], \bigoplus_{i = 0}^nR^i\pi_{M*}\CO_M \otimes \Omega^1_B[-i + 1]\right)
\end{multline}
and respectively
\begin{multline}\label{eq:dec2}
\mathrm{Hom}_B(L0_B^*0_{B*}\pi_{M*}\Omega^1_{M/B}, L0_B^*0_{B*}\Omega_B^1[1]) \\
\simeq{} \mathrm{Hom}_B\left(\bigoplus_{i = 0}^n\wedge^iN_{B/P}^{\vee} \otimes \pi_{M*}\Omega^1_{M/B}[i], \bigoplus_{i = 0}^n\wedge^iN_{B/P}^{\vee} \otimes \Omega^1_B[i + 1]\right).
\end{multline}
As is usual with decomposition-type theorems, the decompositions \eqref{eq:dec1} and \eqref{eq:dec2} themselves are not canonical. Yet one can associate canonical filtrations via the natural truncations $\tau_{\geq j}, \tau_{\leq j}$ of the arguments of $\mathrm{Hom}(-, -)$. In particular, the left-hand side of \eqref{eq:dec1} admits a canonical quotient
\begin{equation} \label{eq:ko}
\mathrm{Hom}_B(R^{n - 1}\pi_{M*}\CO_M \otimes \pi_{M*}\Omega^1_{M/B}[-n + 1], R^n\pi_{M*}\CO_M\otimes \Omega^1_B[-n + 1]).
\end{equation}
Similarly, the left-hand side of \eqref{eq:dec2} admits a canonical quotient
\begin{equation} \label{eq:canquot}
\mathrm{Hom}_B(N_{B/P}^\vee \otimes \pi_{M*}\Omega^1_{M/B}[1], \Omega^1_B[1]).
\end{equation}

By definition, the morphism \eqref{eq:Omega1'} is obtained by applying $L0_B^*$ to the extension class $\epsilon_{\Omega^1_M}$, and then projecting to the quotient space \eqref{eq:canquot}. This in turn corresponds (via $\otimes \omega_B^\vee[n]$) to the morphism obtained by applying~$R\pi_{M*}$ to the extension class of the cotangent sequence~\eqref{eq:cotan}, and then projecting to the quotient space \eqref{eq:ko}. In other words, we have a commutative diagram
\begin{equation} \label{eq:gmseq}
\begin{tikzcd}
\pi_{M*}\Omega^1_{M/B} \arrow{r}{} \arrow{d}{\simeq} & N_{B/P} \otimes \Omega_B^1 \arrow{d}{\simeq} \\
\pi_{M*}\Omega^1_{M/B} \arrow{r}{} \arrow{d}{\simeq} & (R^{n - 1}\pi_{M*}\CO_M)^\vee \otimes R^n\pi_{M*}\CO_M \otimes \Omega^1_B \arrow{d}{\simeq} \\
\pi_{M*}\Omega^1_{M/B} \arrow{r}{} & R^1\pi_{M*}\CO_M \otimes \Omega_B^1.
\end{tikzcd}
\end{equation}
Here the first row is the morphism \eqref{eq:Omega1'}, the second row is obtained from the element in \eqref{eq:ko}, and the third row uses the (fiberwise) perfect pairing
\[
R^1\pi_{M*}\CO_M \otimes R^{n - 1}\pi_{M*}\CO_M \to R^n\pi_{M*}\CO_M
\]
for a smooth abelian fibration $\pi_M: M \to B$.

We set $\CV_k = R^k\pi_{M*}\BQ_M \otimes_\BQ \CO_B$ with graded pieces $\CV^{i, k - i} = \mathrm{gr}^i\CV_k$. The associated graded of the Gauss--Manin connection takes the form
\[
\overline{\nabla}: \CV^{i, k - i} \to \CV^{i - 1, k - i + 1} \otimes \Omega_B^1
\]
and is linear with respect to the $\CV^{0, k - i}$-factor in $\CV^{i, k - i} \simeq \CV^{0, k - i} \otimes \CV^{i, 0}$ by virtue of Griffiths transversality. Finally, by the Katz--Oda description of the Gauss--Manin connection \cite{KO}, the extension class of the cotangent sequence \eqref{eq:cotan} projected to \eqref{eq:ko} is the associated graded
\begin{equation} \label{eq:gm2}
\overline{\nabla}: \CV^{1, n - 1} \to \CV^{0, n} \otimes \Omega_B^1.
\end{equation}
Since \eqref{eq:gm2} is linear with respect to the $\CV^{0, n - 1}$-factor in $\CV^{1, n - 1} \simeq \CV^{0, n - 1} \otimes \CV^{1, 0}$, we find that the bottom row of \eqref{eq:gmseq} is precisely the associated graded
\[
\overline{\nabla}: \CV^{1, 0} \to \CV^{0, 1} \otimes \Omega^1_B.
\]
Note that this time the factor $\Omega^1_B$ is placed on the right as opposed to \eqref{eq:gm}. Also note that we have yet to use the symplectic form $\sigma$ and the $\pi_M$-relatively ample line bundle $\Theta$. To deal with the side-change we evoke the following result of Donagi--Markman~\cite{DM}; see also \cite[Theorem 4.4]{Voisin} and \cite[Lemma 1.2]{SY2}.
\end{proof}

\begin{lem}[Donagi--Markman cubic condition] \label{dm}
Under the isomorphisms
\[
T_B \simeq \pi_{M*}\Omega^1_{M/B}, \quad R^1\pi_{M*}\CO_M \simeq \Omega^1_B
\]
induced by the symplectic form $\sigma$ together with the $\pi_M$-relatively ample line bundle $\Theta$, the associated graded of the Gauss--Manin connection
\[
\overline{\nabla}: \pi_{M*}\Omega^1_{M/B} \to R^1\pi_{M*}\CO_M \otimes  \Omega^1_B
\]
comes from a cubic form in $H^0(B, \mathrm{Sym}^3\Omega^1_B)$.
\end{lem}

\subsection{General case}
We proceed to higher $k$. We show by exploiting functorial properties of smooth abelian fibrations that the higher $k$ cases follow from $k = 1$.

We first extend the Pontryagin product in Section \ref{Sec3.2} to more general (formal) group spaces $\pi_G : G \to B$. Later we will mainly consider $G = P, T^*B, \hat{B}$. Let
\[
m_G: G \times_B G \to G
\]
be the (formal) addition map. We define
\[
\star^R : D\mathrm{QCoh}(G) \times D\mathrm{QCoh}(G) \to D\mathrm{QCoh}(G), \quad (\CE_1, \CE_2) \mapsto Rm_{G*}(q_1^*\CE_1 \otimes^L q_2^*\CE_2)
\]
where $q_1$ and $q_2$ are the two projections from $G \times_B G$. The underived version $\star$ is defined accordingly.

\begin{lem} \label{lem5.4}
Let $\CE = \phi_{\mathrm{FM}}(\Omega^1_M) \in \mathrm{Coh}(P)$ (resp.~$\CE = \mathrm{gr}(P_1) \in \mathrm{Coh}(T^*B)$). For $k \geq 1$, we~have
\[
\CE^{\star^R k} \simeq \CE^{\star k} \in \mathrm{Coh}(P) \quad \textrm{(resp.~$\in \mathrm{Coh}(T^*B)$)}
\]
and
\[
(\CE^{\star k})|_{\hat{B}} \simeq (\CE|_{\hat{B}})^{\star k} \in \mathrm{Coh}(\hat{B}).
\]
\end{lem}

\begin{proof}
Take $\CE_1, \CE_2 \in \mathrm{Coh}(P)$ such that $\mathrm{supp}^{\mathrm{red}}(\CE_i) \subset B$, $i = 1, 2$. Since the reduced support~of
\[
q_1^*\CE_1 \otimes^L q_2^*\CE_2 \simeq q_1^*\CE_1 \otimes q_2^*\CE_2 \in \mathrm{Coh}(P \times_B P)
\]
is contained in the $0$-section $B \subset P \times_B P$, we have
\[
\CE_1 \star^R \CE_2 \simeq Rm_{P*}(q_1^*\CE_1 \otimes q_2^*\CE_2) \simeq m_{P*}(q_1^*\CE_1 \otimes q_2^*\CE_2) \simeq \CE_1 \star \CE_2 \in \mathrm{Coh}(P)
\]
with $\mathrm{supp}^{\mathrm{red}}(\CE_1 \star \CE_2 ) \subset B$. We also have
\begin{align*}
(\CE_1 \star \CE_2)|_{\hat{B}} & \simeq (m_{P*}(q_1^*\CE_1 \otimes q_2^*\CE_2))|_{\hat{P}} \\
& \simeq m_{\hat{B}*}((q_1^*\CE_1 \otimes q_2^*\CE_2)|_{\hat{B}}) \\
& \simeq m_{\hat{B}*}(q_1^*(\CE_1|_{\hat{B}}) \otimes q_2^*(\CE_2|_{\hat{B}})) \\
& \simeq \CE_1|_{\hat{B}} \star \CE_2|_{\hat{B}} \in \mathrm{Coh}(\hat{B}).
\end{align*}
Here the second isomorphism uses the base change with respect to $m_P: \mathrm{supp}(q_1^*\CE_1 \otimes q_2^*\CE_2) \to P$, and the third isomorphism uses the flatness of $|_{\hat{B}}$. The statement for $\CE = \phi_{\mathrm{FM}}(\Omega^1_M)$ follows by induction on $k$ and the proof for $\CE = \mathrm{gr}(P_1)$ is identical.
\end{proof}

Lemma \ref{lem5.4} together with Propositions \ref{prop3.3} and \ref{prop5.2} implies that
\begin{equation} \label{eq:five}
\phi_{\mathrm{FM}}((\Omega_M^1)^{\otimes k})|_{\hat{B}} \simeq (\phi_{\mathrm{FM}}(\Omega_M^1)^{\star k})|_{\hat{B}} \simeq (\phi_{\mathrm{FM}}(\Omega_M^1)|_{\hat{B}})^{\star k} \simeq (\mathrm{gr}(P_1)|_{\hat{B}})^{\star k} \simeq (\mathrm{gr}(P_1)^{\star k})|_{\hat{B}}.
\end{equation}
To understand $\mathrm{gr}(P_1)^{\star k}$, we recall the Higgs interpretation of $\mathrm{gr}(P_1)$. Let $p: T^*B \to B$ be the natural projection. The pushforward $p_*\mathrm{gr}(P_1)$ is a $p_*\CO_{T^*B} \simeq \mathrm{Sym}(T_B)$-module, which is precisely given by the vector bundle
\[
\mathrm{gr}(\CV_1) := \CV^{0, 1} \oplus \CV^{1, 0}
\]
together with the Higgs field
\[
\overline{\nabla}|_{\CV^{0, 1}} = 0, \quad \overline{\nabla}|_{\CV^{1, 0}}: \CV^{1,0} \to \Omega^1_B \otimes \CV^{0, 1}.
\]
Conversely, from the Higgs bundle $(\mathrm{gr}(\CV_1), \overline{\nabla})$ one also recovers the $\CO_{T^*B}$-module
\[
\mathrm{gr}(P_1) \simeq p^{-1} \mathrm{gr}(\CV_1) \otimes_{p^{-1}p_*\CO_{T^*B}} \CO_{T^*B}.
\]

\begin{lem} \label{lem5.5}
For $k \geq 1$, we have
\[
p_*(\mathrm{gr}(P_1)^{\star k}) \simeq (\mathrm{gr}(\CV_1)^{\otimes k}, \overline{\nabla}^{\otimes k}) \in \mathrm{Coh}(p_*\CO_{T^*B}).
\]
Here the Higgs field $\overline{\nabla}^{\otimes k}$ on $\mathrm{gr}(\CV_1)^{\otimes k}$ is defined by the Leibniz rule.
\end{lem}

\begin{proof}
The addition map $m_{T^*B}: T^*B \times_B T^*B \to T^*B$ corresponds to the morphism of sheaves
\begin{equation} \label{eq:hopf}
p_*\CO_{T^*B} \to p_*\CO_{T^*B} \otimes_{\CO_B} p_*\CO_{T^*B}.
\end{equation}
In local coordinates $\CO_U[\xi_1, \ldots, \xi_n]$ over $U \subset B$, the map \eqref{eq:hopf} sends $\xi_j$ to $\xi_j \otimes 1 + 1 \otimes \xi_j$. 

Take $\CF_1, \CF_2 \in \mathrm{Coh}(T^*B)$ such that $p_*\CF_i \simeq (\CE_i, \theta_i) \in \mathrm{Coh}(p_*\CO_{T^*B})$, $i = 1, 2$. Then the $p_*\CO_{T^*B}$-module $p_*(\CF_1 \star \CF_2)$ is given by the tensor product $\CE_1 \otimes \CE_2$ whose $p_*\CO_{T^*B}$-module structure is obtained by composing with \eqref{eq:hopf}, hence the Leibniz rule for the Higgs fields $\theta_1$ and $\theta_2$. The statement of the lemma follows by induction on $k$.
\end{proof}

We are ready to prove Conjecture \ref{main_conj2} for the smooth Lagrangian fibration $\pi_M : M \to B$.

\begin{proof}[Proof of Theorem \ref{thm0.4}]
As in the proof of Proposition \ref{prop5.2} we set $\CV_k = R^k\pi_{M*}\BQ_M \otimes_\BQ \CO_B$ with graded pieces $\CV^{i, k - i} = \mathrm{gr}^i\CV_k$. The associated graded $\mathrm{gr}(P_k) \in \mathrm{Coh}(T^*B)$ corresponds via $p_*$ to the Higgs bundle
\[
\left(\mathrm{gr}(\CV_k) := \bigoplus_{i = 0}^k\CV^{i, k - i}, \overline{\nabla}\right) \in \mathrm{Coh}(p_*\CO_{T^*B})
\]
where
\[
\overline{\nabla}|_{\CV^{i, k - i}}: \CV^{i, k - i} \to \Omega_B^1 \otimes \CV^{i - 1, k - i + 1}
\]
is the associated graded of the Gauss--Manin connection. Now since $\pi_M: M \to B$ is a smooth abelian fibration, we have canonical isomorphisms
\[
\CV_k \simeq \wedge^k\CV_1, \quad \CV^{i, k - i} \simeq \CV^{0, k - i} \otimes \CV^{i, 0} \simeq \wedge^{k - i}\CV^{0, 1} \otimes \wedge^i\CV^{1, 0},
\]
and
\begin{equation} \label{eq:wedge}
(\mathrm{gr}(\CV_k), \overline{\nabla}) \simeq (\wedge^k\mathrm{gr}(\CV_1), \wedge^k\overline{\nabla})
\end{equation}
for all $k \geq 1$. Here the Higgs field $\wedge^k\overline{\nabla}$ on $\wedge^k\mathrm{gr}(\CV_1)$ is defined by the Leibniz rule.

By \eqref{eq:five} the restrictions of $\phi_{\mathrm{FM}}((\Omega_M^1)^{\otimes k})$ and $\mathrm{gr}(P_1)^{\star k}$ to the formal neighborhood(s) $\hat{B}$ are isomorphic. Moreover $\mathrm{gr}(P_1)^{\star k}$ corresponds via $p_*$ to the Higgs bundle $(\mathrm{gr}(\CV_1)^{\otimes k}, \overline{\nabla}^{\otimes k})$ by Lemma~\ref{lem5.5}. It is also clear that both Lemmas \ref{lem5.4} and~\ref{lem5.5} are $\mathfrak{S}_n$-equivariant. Taking antisymmetric tensors and in view of \eqref{eq:wedge}, we conclude that
\[
\phi_{\mathrm{FM}}(\Omega_M^k)|_{\hat{B}} \simeq \mathrm{gr}(P_k)|_{\hat{B}}. \qedhere
\]
\end{proof}

\section{Two-dimensional cases}

\subsection{Overview} \label{sec:6.1}
In this section, we prove Theorem \ref{thm0.5}. By our assumption, the morphism
\[
\pi_M: M \to B
\]
is an elliptic fibration over a non-proper curve $B$ with integral fibers and a section $s_B: B \to M$. Moreover, the closed fibers of $\pi_M$ are either a nonsingular elliptic curve, or a nodal rational curve. Assume that $p_1, \cdots, p_m$ are the closed points on $B$ such that the fiber $F_i = \pi_M^{-1}(p_i)$ is nodal, and the restriction
\[
\pi^\circ:= \pi_{M^\circ}: M^\circ  \to B^\circ:= B \setminus \{p_1,\dots,p_m\}, \quad M^\circ:= \pi^{-1}(B^\circ)
\]
is a smooth elliptic fibration. Denote by $j: B^\circ \hookrightarrow B$ the natural open embedding, and denote by $\CV$ the variation of Hodge structures given by $R^1\pi^\circ_* \BQ_M$ over $B^\circ$ as in Section \ref{smooth}. Then a direct calculation (\emph{e.g.}~Proof of \cite[Proposition 4.17]{Z}) yields
\[
P_1 = {j_{!*}}\CV,
\]
where we use the same notation $\CV$ to denote the (pure) Hodge module given by the flat bundle $(\CV, \nabla)$.

In view of Corollary \ref{cor3.2}, Theorem \ref{thm0.5} is reduced to the following theorem.

\begin{thm}\label{thm6.1}
Under the isomorphism $\hat{\kappa}_\Theta$ of Proposition \ref{prop1.4} identifying the formal neighborhoods of $B$ in $P$ and $T^*B$, we have 
\[
\phi_{\mathrm{FM}}(\Omega_M^1)|_{\hat{B}} \simeq \mathrm{gr}({j_{!*}}\CV)|_{\hat{B}} \in \mathrm{Coh}(\hat{B}).
\]
\end{thm}

\subsection{Gauss--Manin with poles}


We choose as in Section \ref{Sec5.2} a symplectic form $\sigma$ on $M$ as well as a $\pi_M$-relative ample bundle $\Theta$. They induce isomorphisms 
\begin{equation}\label{isom1}
R^1\pi_{M*} \CO_M \simeq \Omega^1_B, \quad N_{B/P} \simeq \Omega^1_B
\end{equation}
where the first isomorphism is given by \cite{Ma3}. We denote by $D$ the effective divisor $\sum_{i=1}^m p_i \subset B$. The divisor
\[
{F} : = \sum_{i=1}^m F_i \subset M
\]
given by the pullback of $D$ is normal crossing on $M$. The symplectic form $\sigma$ also induces an isomorphism
\begin{equation}\label{6.2_1}
\Omega^1_{M/B}(\log {F})  \simeq \Omega_B^{1\vee}.
\end{equation}
By the discussion of Section \ref{Sec5.3}, if we restrict over $B^\circ$, the fixed isomorphisms (\ref{isom1}) and (\ref{6.2_1}) and the Donagi--Markman cubic form induce an isomorphism
\[
 \phi_{\mathrm{FM}}(\Omega_M^1) |_{\hat{B}^\circ} \xrightarrow{\simeq} \mathrm{gr}(\CV)|_{\hat{B}^\circ}.
\]
Note that the symmetry of the Donagi--Markman cubic form is automatic in this case since the base $B$ is 1-dimensional.

Now in order to extend the isomorphism above over $D$, we first need to extend the associated graded of the Gauss--Manin connection, as well as the Donagi--Markman cubic form. We consider the logarithmic cotangent sequence on $M$:
\begin{equation}\label{6.2_0}
0 \to \pi_M^*\Omega_B^1(D) \to \Omega^1_M(\log {F}) \to \Omega^1_{M/B}(\log {F}) \to 0.
\end{equation}
After pushing it forward to $B$, we have the connecting map 
\[
\pi_{M*} \Omega^1_{M/B}(\log {F}) \to   R^1\pi_{M*} \CO_M \otimes \Omega^1_B(D).
\]
We review briefly the Hodge theoretic interpretation of this map via Deligne's canonical extension of $\CV$; see \cite[Section 2]{K}. Recall that the canonical extension depends on a real interval $[a, a+1)$ or $(a, a + 1]$ where the eigenvalues of the residue endomorphism should lie. In our situation, the monodromy around each point of $D$ is unipotent (given by the matrix $(\begin{smallmatrix} 1&1\\0&1 \end{smallmatrix})$ in local coordinates), so the eigenvalues are necessarily integers. Let $\overline{\CV}$ be the canonical extension of $\CV$ with respect to either $[0, 1)$ or~$(-1, 0]$; it is locally free of rank $2$ on $B$. Schmid's theorem says that~$F_\bullet \overline{\CV} := j_*F_\bullet \CV \cap \overline{\CV}$ is a filtration by locally free subsheaves; here we write the Hodge filtration as an increasing filtration $F_i\CV: = F^i\CV$ to be compatible with the convention of Hodge modules,
\[
(\overline{\CV}, F^\bullet, \nabla), \quad \nabla: \overline{\CV} \to   \overline{\CV}\otimes \Omega^1_B(D),\quad  F_{-1}\overline{\CV}  \subset F_{0}\overline{\CV} = \overline{\CV}.
\]
Then by the logarithmic version of the Katz--Oda theorem \cite{Katz}, the connecting map above associated with (\ref{6.2_0}) recovers the associated graded of the meromorphic Gauss--Manin connection $\overline{\nabla}: \mathrm{gr}_{-1}\overline{\CV} \to \mathrm{gr}_{0}\overline{\CV} \otimes \Omega^1_B(D)$:
\begin{equation}\label{6.2_2}
\overline{\nabla}: \pi_{M*} \Omega^1_{M/B}(\log {F}) \to   R^1\pi_{M*} \CO_M \otimes \Omega^1_B(D).
\end{equation}
Further using the isomorphisms (\ref{isom1}) and (\ref{6.2_1}), we obtain that (\ref{6.2_2}) comes from a section
\begin{equation}\label{cubic2}
[\overline{\nabla}] \in H^0(B, (\Omega^1_B)^{\otimes 3}(D)).
\end{equation}
This is indeed the meromorphic extension of the Donagi--Markman cubic form in Lemma \ref{dm}.

\subsection{Admissible sheaves}

Our strategy is similar to the proof of Theorem \ref{thm0.4}. We express both sheaves 
\[
\phi_{\mathrm{FM}}(\Omega^1_M)|_{\hat{B}} \quad \textup{and} \quad \mathrm{gr}(j_{!*}\CV)|_{\hat{B}}
\]
in terms certain ``building blocks'' supported scheme-theoretically on either the 0-section \mbox{$B \subset \hat{B}$} or a closed fiber $\hat{F}_i : = F_i|_{\hat{B}}$, and then we match their extension classes. Since the existence of the singular fibers $F_i$ further complicates the extensions, we introduce the notion of \emph{admissible sheaves} to treat such complexity.

For notational convenience, we will uniformly use $F_i$ to denote the fiber over $p_i$ for either $\pi_M: M \to B$ or the projection $T^*B \to B$. Therefore $F_i$ is either a nodal rational curve or the affine line $\BC$.

We say that an object 
\[
\CA \in \mathrm{Coh}(T^*B) \quad (\textrm{resp.~}\CA \in \mathrm{Coh}(\hat{B}))
\]
is an \emph{admissible sheaf}, if $\CA$ admits an increasing filtration of coherent subsheaves
\[
\CW^\CA_{-1} \subset \CW^{\CA}_0 \subset \CA 
\]
satisfying that 
\begin{enumerate}
    \item[(a)] $\CW^\CA_{-1} \simeq \oplus_{i=1}^m \CO_{F_i}$ (resp.~$\CW^\CA_{-1} \simeq \oplus_{i=1}^m \CO_{\hat{F}_i}$),
    \item[(b)] $\CW^\CA_0/ \CW^\CA_{-1} \simeq 0_{B*}\Omega^1_B(D)$, and
    \item[(c)] $\CA/\CW^\CA_0 \simeq 0_{B*}\Omega_B^{1\vee}$.
\end{enumerate}
By (a, b), the subsheaf $\CW_0^\CA$ fits into a short exact sequence\footnote{For notational convenience, we only describe here the case where $\CA \in \mathrm{Coh}(T^*B)$; the case for $\hat{B}$ is completely parallel.}
\begin{equation}\label{eqn56}
0 \to \bigoplus_{i=1}^m \CO_{F_i} \to \CW_0^\CA \to 0_{B*} \Omega^1_B(D) \to 0,
\end{equation}
whose extension class induces
\begin{equation}\label{eqn57}
[\CW_0^{\CA}] \in \mathrm{Ext}_{T^*B}^1\left(0_{B*}\Omega^1_B(D), \bigoplus_{i=1}^m \CO_{F_i}\right) = \bigoplus_{i=1}^m \mathrm{Ext}^1_{T^*B}(0_{B*} \Omega^1_B(D), \CO_{F_i}).
\end{equation}
By adjunction, each extension group $\mathrm{Ext}^1_{T^*B}(0_{B*} \Omega^1_B(D), \CO_{F_i})$ on the right-hand side of (\ref{eqn57}) is 1-dimensional:
\[
\mathrm{Ext}^1_{T^*B}(0_{B*} \Omega^1_B(D), \CO_{F_i})  = \mathrm{Ext}^1_{F_i}(\BC_{p_i}, \CO_{F_i})\simeq \BC.
\]
Here $p_i$ is viewed as a point on $F_i$ lying in the intersection with the 0-section $B \subset T^*B$. We say that the admissible sheaf $\CA$ is \emph{good} if each summand of the extension class $[\CW_0^\CA]$ in~$\mathrm{Ext}^1_{T^*B}(0_{B*} \Omega^1_B(D), \CO_{F_i})$ is nonzero. In this case, up to scaling $\CO_{F_i}$ we may express~$\CW^{\CA}_0$ as an extension (\ref{eqn56}) whose extension class (\ref{eqn57}) is of the form
\[
(1,1, \dots, 1) \in  \bigoplus_{i=1}^m \mathrm{Ext}^1_{T^*B}(0_{B*} \Omega^1_B(D), \CO_{F_i}).
\]
Now assume that $\CA$ is a good admissible sheaf. The condition (c) further implies that $\CA$ fits into an extension
\[
0 \to \CW^\CA_0 \to \CA \to 0_{B*} \Omega_B^{1\vee} \to 0
\]
which yields a class
\[
[\CA] \in \mathrm{Ext}^1_{T^*B}(0_{B*}\Omega_B^{1\vee}, \CW_0^\CA) \xrightarrow{\rho_\CA} \mathrm{Ext}^1_{T^*B}(0_{B*}\Omega_B^{1\vee}, 0_{B*}\Omega^1_B(D)). 
\]
Here the map $\rho_\CA$ is induced by (\ref{eqn56}).

The following proposition provides a criterion for two good admissible sheaves to be isomorphic. This is also the only place where we need the assumption that $B$ is not proper.

\begin{prop}\label{prop6.2}
Assume $B$ non-proper. Let $\CA$ and $\CA'$ be two good admissible sheaves such that the classes $\rho_{\CA}([\CA])$ and $\rho_{\CA'}([\CA'])$ coincide:
\[
\rho_\CA([\CA]) =  \rho_{\CA'}([\CA'])\in  \mathrm{Ext}^1_{T^*B}(0_{B*}\Omega_B^{1\vee}, 0_{B*}\Omega^1_B(D)).
\]
Then we have
\[\CA \simeq \CA'.
\]
\end{prop}

\begin{proof}
Applying $\mathrm{Hom}_{T^*B}(0_{B*}\Omega^{1\vee}_B, - )$ to (\ref{eqn56}), we obtain the long exact sequence:
\begin{align*}
 \cdots    \rightarrow \mathrm{Hom}_{T^*B}(0_{B*}\Omega_B^{1\vee}, 0_{B*}\Omega^1_B(D)) \xrightarrow{\textup{(i)}} \bigoplus_{i=1}^m \mathrm{Ext}_{T^*B}^1(0_{B*}\Omega^{1\vee}_B, \CO_{F_i}) &  \\ \xrightarrow{\textup{(ii)}} \mathrm{Ext}^1_{T^*B}(0_{B*}\Omega_B^{1\vee}, \CW_0^\CA) \xrightarrow{\rho_{\CA}}  \mathrm{Ext}^1_{T^*B}(0_{B*}\Omega_B^{1\vee}, 0_{B*}\Omega^1_B(D)).& 
\end{align*}
We would like to show in the following that, for a good admissible sheaf $\CA$, the map (i) in the chain above is surjective, so that (ii) is 0. In particular, the extension class $[\CA]$ is completely characterized by its image $\rho_\CA([\CA])$ and the proposition follows.

To prove the surjectivity of (i), we consider the residue exact sequence
\[
 0 \to \Omega^1_B \to \Omega^1_B(D) \xrightarrow{\mathrm{res}} \bigoplus_{i=1}^m \BC_{p_i} \to 0.
\]
Applying $\mathrm{Hom}_B(\Omega_B^{1\vee}, - )$, this induces the long exact sequence
\[
\cdots \to \mathrm{Hom}_B(\Omega_B^{1\vee}, \Omega^1_B(D)) \xrightarrow{\textup{(i)'}}  \BC^{\oplus m} \to \mathrm{Ext}_B^1(\Omega_B^{1\vee}, \Omega^1_B) \to \cdots
\]
Using the fact that $B$ is not proper, we obtain that 
\[
\mathrm{Ext}_B^1(\Omega_B^{1\vee}, \Omega^1_B) = 0.
\]
Therefore the map (i)' above is surjective. On the other hand, the map (i)' recovers (i):
\[
\mathrm{Hom}_{T^*B}(0_{B*}\Omega_B^{1\vee}, 0_{B*}\Omega^1_B(D)) = \mathrm{Hom}_B(\Omega_B^{1\vee}, \Omega^1_B(D)) \xrightarrow{\textup{(i)} = \textup{(i)'}} \BC^{\oplus m} = \bigoplus_{i=1}^m \mathrm{Ext}_{T^*B}^1(0_{B*}\Omega^{1\vee}_B, \CO_{F_i}).
\]
Hence we conclude the surjectivity of (i), which completes the proof of the proposition.
\end{proof}

In the following three sections, we prove Theorem \ref{thm6.1} by showing that both sheaves obtained from Fourier--Mukai and the Hodge module theory respectively are good admissible sheaves, and their classes in the group 
\[
\mathrm{Ext}^1_{\hat{B}}(0_{B*}\Omega_B^{1\vee}, 0_{B*} \Omega^1_B(D))
\]
are both essentially governed by the cubic form (\ref{cubic2}).

\subsection{Hodge modules}\label{Sec6.4}
We first treat the Hodge module side. The Hodge module $j_{!*}\CV$ on $B$ can be described concretely using the canonical extension $(\overline{\CV}, F^\bullet, \nabla)$. More precisely, by \cite[3.10]{S2} we have
\[
j_{!*}\CV = \CD_B \cdot \overline{\CV} \subset \overline{\CV}(*D)
\]
where the $\CD_B$-action is induced by Deligne's meromorphic connection on $\overline{\CV}(*D)$, and
\begin{equation}\label{eq:ic}
F_kj_{!*}\CV = \sum_{i \geq 0} F_i\CD_B \cdot F_{k - i}\overline{\CV}.
\end{equation}
Each associated graded object $\mathrm{gr}_i(j_{!*}\CV) = F_ij_{!*}\CV /F_{i-1} j_{!*}\CV$ is a coherent sheaf on $B$. The coherent sheaf
\[
\mathrm{gr}(j_{!*}\CV) \in \mathrm{Coh}(T^*B)
\]
obtained from the Hodge module $j_{!*}\CV$ is then completely described by the quasi-coherent sheaf
\[
\bigoplus_{k\geq -1} \mathrm{gr}_k(j_{!*}\CV)\in \mathrm{QCoh}(B)
\]
together with the Higgs field 
\begin{equation}\label{Higgs_fields}
\overline{\nabla}: \mathrm{gr}_k(j_{!*}\CV) \to \mathrm{gr}_{k+1}(j_{!*}\CV)\otimes \Omega_B^1
\end{equation}
where by Griffiths transversality all the nontrivial Higgs fields only increase the index by 1. Following a direct calculation using the formulas above as in \cite[Section 2.4]{SY2}, we have that
\[
\mathrm{gr}_{-1}(j_{!*}\CV) \simeq \Omega_B^{1\vee}, \quad \mathrm{gr}_{0}(j_{!*}\CV) \simeq \Omega^1_B(D), \quad \mathrm{gr}_{k}(j_{!*}\CV) \simeq \bigoplus_{i=1}^m \BC_{p_i}, \quad k>0.
\]
Moreover the nontrivial Higgs fields (\ref{Higgs_fields}) are given by (\ref{cubic2}) for $k=-1$, the nontrivial residue map
\[
\Omega^1_B(D) \xrightarrow{\mathrm{res}} \bigoplus_{i=1}^m \BC_{p_i}
\]
for $k=0$, and the identity maps 
\[
\mathrm{id}: \bigoplus_{i=1}^m \BC_{p_i} \xrightarrow{\simeq } \bigoplus_{i=1}^m \BC_{p_i}.
\]
for all $k >0$. This allows us to express the object
\[
\mathrm{gr}(j_{!*}\CV) \in \mathrm{Coh}(T^*B)
\]
as a good admissible sheaf 
\[
\CW_{-1}^{\mathrm{HM}} \subset \CW_{0}^{\mathrm{HM}} \subset \mathrm{gr}(j_{!*}\CV).
\]
Here the subsheaf $\CW_i^{\mathrm{HM}} \in \mathrm{Coh}(T^*B)$ is given by
\[
\bigoplus_{k \geq -i} \mathrm{gr}_k(j_{!*}\CV)
\]
endowed with the restricted Higgs field
\[
\overline{\nabla}: \bigoplus_{k \geq -i} \mathrm{gr}_k(j_{!*}\CV) \to \left(\bigoplus_{k \geq -i + 1} \mathrm{gr}_k(j_{!*}\CV) \right) \otimes \Omega^1_B.
\]
Finally, an identical argument of Section \ref{Sec5.3} yields a splitting
\begin{equation}\label{splitting_HM}
\mathrm{Hom}_{T^*{B}}( 0_{B*} \Omega_B^{1\vee}, 0_{B*} \Omega^1_B(D)[1] ) = \mathrm{Hom}_B(\Omega_B^{1\vee}, \Omega^1_B(D)[1]) \oplus H^0(B, (\Omega^1_B)^{\otimes 3}(D)).
\end{equation}
The class 
\[
\rho_{\mathrm{HM}}([\mathrm{gr}(j_{!*}\CV)]) \in \mathrm{Hom}_{T^*{B}}( 0_{B*} \Omega_B^{1\vee}, 0_{B*} \Omega^1_B(D)[1] )
\]
is then $(0, [\overline{\nabla}])$ via the splitting (\ref{splitting_HM}) with $[\overline{\nabla}]$ given by (\ref{cubic2}).

\subsection{Fourier--Mukai}\label{Sec6.5}


For our elliptic fibration $\pi_M: M \to B$ with a section $s_B: B \to M$, the partial Fourier--Mukai transform 
\[
\phi_{\mathrm{FM}}: D^b\mathrm{Coh}(M) \to D^b\mathrm{Coh}(P)
\]
can actually be upgraded to a \emph{full} Fourier--Mukai transform 
\[
\widetilde{\phi}_{\mathrm{FM}}: D^b\mathrm{Coh}(M) \to D^b\mathrm{Coh}(M)
\]
as we review in the following. 

We denote by $M^\vee \to B$ the relative compactified Jacobian fibration parameterizing torsion-free, rank 1, degree 0 sheaves on the fibers of $\pi: M \to B$. Since $M^\vee$ is a fine moduli space with a $0$-section $0_B: B \rightarrow M^\vee$, there is a universal Poincar\'e sheaf on $M \times_B M^\vee$. We choose the universal Poincar\'e sheaf $\CP$ such that the induced Fourier--Mukai transform
\[
\widetilde{\phi}_{\mathrm{FM}}: D^b\mathrm{Coh}(M) \xrightarrow{\simeq} D^b\mathrm{Coh}(M^\vee)
\]
satisfies $\widetilde{\phi}_{\mathrm{FM}}(\CO_M) \simeq 0_{B*} \CO_B$. We have that $M^\vee$ is naturally isomorphic to $M$ by
\[
M \xrightarrow{\simeq} M^\vee, \quad x \mapsto \iota_{s*}(\mathfrak{m}^\vee_x \otimes \CO_{M_s}(-s))
\]
where $M_s$ is the closed fiber containing $x$, $\iota_s: M_s \hookrightarrow M$ is the closed embedding, $\mathfrak{m}_x$ is the ideal sheaf of $x$ in $M_s$, and $s$ is viewed as a point on $M_s$ lying in the intersection with the section $B \subset M$. From now on we identify $M^\vee$ with $M$. The group scheme $P \to B$ is then obtained as an open surface $P\subset M$ removing the nodes of the singular fibers. The partial Fourier--Mukai transform $\phi_{\mathrm{FM}}$ is the composition of the full Fourier--Mukai $\widetilde{\phi}_{\mathrm{FM}}$ with the restriction map associated with $P \subset M$.

\begin{lem}\label{lem6.3} The following hold for the Fourier--Mukai transform $\widetilde{\phi}_{\mathrm{FM}}$.
\begin{enumerate}
    \item[(i)] For $\CK \in D^b\mathrm{Coh}(B)$, we have
    \[
    \widetilde{\phi}_{\mathrm{FM}}(\pi_M^* \CK) \simeq 0_{B*} \CK.
    \]
    \item[(ii)] For a point $x \in M_s$ in a closed fiber $\iota_s: M_s \hookrightarrow M$, we have
    \[
    \widetilde{\phi}_{\mathrm{FM}}(\BC_x) \simeq \iota_{s*} (\mathfrak{m}_x \otimes \CO_{M_s}(s))[1].
    \]
    \item[(iii)] Let $F_i$ be the nodal fiber of $\pi: M \to B$ over $p_i$ with $x_i \in F_i$ the node. Let $\nu_i: \widetilde{F_i} \to F_i$ be the normalization. Then we have
    \[
    \widetilde{\phi}_{\mathrm{FM}}(\nu_{i*} \CO_{\widetilde{F}_i}) \simeq \iota_{p_i*} \mathfrak{m}_{x_i}[1].
    \]
\end{enumerate}
\end{lem}

\begin{proof}
(i, ii) are easy exercises deduced from the definition. We now prove (iii). We consider the short exact sequence
\begin{equation}\label{eqn60}
0 \to \CO_{F_i} \to \nu_{i*} \CO_{\widetilde{F}_i} \to \BC_{x_i} \to 0.
\end{equation}
Applying the functor $\widetilde{\phi}_{\mathrm{FM}}$, we obtain from (i, ii) the exact triangle
\begin{equation}\label{eqn61}
\BC_{p_i} \to \widetilde{\phi}_{\mathrm{FM}}(\nu_{i*}\CO_{\widetilde{F}_i}) \to  \iota_{p_i*} (\mathfrak{m}_{x_i} \otimes \CO_{F_i}(p_i))[1] \xrightarrow{+1}.
\end{equation}
The associated long exact sequence reads:
\[
0 \to \CH^{-1}\left( \widetilde{\phi}_{\mathrm{FM}}(\nu_{i*}\CO_{\widetilde{F}_i})  \right)  \to\iota_{p_i*} (\mathfrak{m}_{x_i} \otimes \CO_{F_i}(p_i)) \to \BC_{p_i} \xrightarrow{(*)} \CH^{0}\left( \widetilde{\phi}_{\mathrm{FM}}(\nu_{i*}\CO_{\widetilde{F}_i})  \right) \to 0.
\]
It suffices to show that the map $(*)$ in the above sequence is trivial. Assume it is not. Then the arrow $\iota_{p_i*} (\mathfrak{m}_{x_i} \otimes \CO_{F_i}(p_i)) \to \BC_{p_i}$ has to be trivial which forces (\ref{eqn61}) to split; equivalently~(\ref{eqn60}) splits which is a contradiction. This completes the proof of (iii).
\end{proof}

Now we show that $\phi_{\mathrm{FM}}(\Omega_M^1)$ is a good admissible sheaf after restricting to a formal neighborhood of the $0$-section $B\subset M$. Instead we work with the object
\begin{equation*}\label{FM_shf}
\widetilde{\phi}_{\mathrm{FM}}(\Omega_M^1) \in D^b\mathrm{Coh}(M).
\end{equation*}
We consider the following triangle of morphisms
\begin{equation}\label{triangle1}
    \begin{tikzcd}[column sep=small]
    \Omega_M^1 \arrow[dr, ""] \arrow[rr, ""] & & \Omega_M^1(\log F) \arrow[dl, ""] \\
       & \Omega^1_{M/B}(\log F)  & 
\end{tikzcd}
\end{equation}
where $\Omega_M^1 \rightarrow \Omega_M^1(\log F)$ and $\Omega_M^1(\log F) \to \Omega^1_{M/B}(\log F)$ are the natural maps, and the third arrow is the composition of these two maps. By the octahedral axiom of triangulated categories, the cones associated with the three maps of (\ref{triangle1}) form an exact triangle
\[
\bigoplus_{i=1}^m \nu_{i*} \CO_{\widetilde{F}_i}[-1] \to \CK_0 \to \pi_{M}^* \Omega^1_B(D) \xrightarrow{+1}.
\]
Here $\CK_0:= \mathrm{cone}\left( \Omega_{M}^1 \to \Omega_{M/B}^1(\log F)\right)[-1]$. Applying the functor $\widetilde{\phi}_{\mathrm{FM}}$ to this exact triangle, we obtain from Lemma \ref{lem6.3} that the object $\widetilde{\phi}_{\mathrm{FM}}(\CK_0)$ is a sheaf concentrated in degree 0 which fits into the exact sequence
\begin{equation}\label{eq64}
0 \to \bigoplus_{i=1}^m \iota_{p_i*}\mathfrak{m}_{x_i} \to \widetilde{\phi}_{\mathrm{FM}}(\CK_0) \to 0_{B*} \Omega^1_B(D) \to 0.
\end{equation}
We set 
\[
\widetilde{\CW}_{-1}^{\mathrm{FM}}: = \bigoplus_{i=1}^m \iota_{p_i*}\mathfrak{m}_{x_i}, \quad \widetilde{\CW}_0^{\mathrm{FM}}:= \widetilde{\phi}_{\mathrm{FM}}(\CK_0).
\]

\begin{prop}\label{prop6.4}
The object $\widetilde{\phi}_{\mathrm{FM}}(\Omega_M^1)$ is a coherent sheaf on $M$ with an increasing filtration of subsheaves:
\[
\widetilde{\CW}_{-1}^{\mathrm{FM}} \subset \widetilde{\CW}_0^{\mathrm{FM}} \subset \widetilde{\phi}_{\mathrm{FM}}(\Omega_M^1) 
\]
which satisfies
\begin{enumerate}
    \item[(a)] $\widetilde{\CW}_{-1}^{\mathrm{FM}} = \bigoplus_{i=1}^m \iota_{p_i*}\mathfrak{m}_{x_i}$,  \item[(b)] $\widetilde{\CW}^\mathrm{FM}_0/ \widetilde{\CW}^\mathrm{FM}_{-1} \simeq 0_{B*}\Omega^1_B(D)$, and
    \item[(c)] $\widetilde{\phi}_{\mathrm{FM}}(\Omega_M^1)/\widetilde{\CW}^\mathrm{FM}_0 \simeq 0_{B*}\Omega_B^{1\vee}$.
\end{enumerate}
\end{prop}

\begin{proof}
By the definition of $\CK_0$, the isomorphism (\ref{6.2_1}), and Lemma \ref{lem6.3} (i), we obtain an exact triangle
\[
\widetilde{\phi}_{\mathrm{FM}}(\Omega_M^1)  \to 0_{B*}\Omega_B^{1\vee} \to  \widetilde{\phi}_{\mathrm{FM}}(\CK_0)[1] \xrightarrow{+1}.
\]
Since $\widetilde{\CW}_0^{\mathrm{FM}} = \widetilde{\phi}_{\mathrm{FM}}(\CK_0)$ is a sheaf concentrated in degree 0, the above exact triangle yields a short exact sequence
\[
0 \to \widetilde{\CW}_0^{\mathrm{FM}} \to  \widetilde{\phi}_{\mathrm{FM}}(\Omega_M^1) \to  0_{B*}\Omega_B^{1\vee} \to 0.
\]
This proves (c). For the remaining parts, we note that (a) is given by the definition of $\widetilde{\CW}_{-1}^{\mathrm{FM}}$, and (b) follows from the short exact sequence (\ref{eq64}).
\end{proof}

Analogously to (\ref{eqn56}), the description of $\widetilde{\CW}^\mathrm{FM}_0$ gives an extension (\ref{eq64}) which induces a class
\[
[\widetilde{\CW}^\mathrm{FM}_0] \in \bigoplus_{i=1}^m \mathrm{Ext}_M^1(0_{B*} \Omega^1_B(D), \iota_{p_i*}\mathfrak{m}_{x_i}).
\]
Each extension group on the right-hand side is 1-dimensional.

\begin{prop}\label{prop6.5}
Each summand of the extension class $[\widetilde{\CW}^\mathrm{FM}_0]$ in
\[
\mathrm{Ext}_M^1(0_{B*} \Omega^1_B(D), \iota_{p_i*}\mathfrak{m}_{x_i})
\]
is nonzero.
\end{prop}

\begin{proof}
Since this is a local question, we may assume that $\pi: M \to B$ only has one singular nodal fiber $F \subset M$ over $p \in B$ with $x\in F$ the node. We need to show that the extension 
\[
0 \to \iota_{p*}\mathfrak{m}_{x}\to \widetilde{\CW}^{\mathrm{FM}}_0 \to 0_{B*}  \Omega^1_B(P) \to 0, \quad \widetilde{\CW}^{\mathrm{FM}}_0  =  \widetilde{\phi}_{\mathrm{FM}}(\CK_0)
\]
is nontrivial. Assume this is trivial. Then applying the inverse Fourier--Mukai $\widetilde{\phi}^{-1}_{\mathrm{FM}}$ to this exact sequence, we obtain that $\CK_0$ has nontrivial cohomology in degrees 0 and 1 with 
\begin{equation}\label{contra}
\CH^1(\CK_0) \simeq \nu_* \CO_{\widetilde{F}}.
\end{equation}
Here $\nu: \widetilde{F} \to F$ is the normalization.

To reach a contradiction, now we describe $\CK_0$ by another triangle of morphisms
\begin{equation}\label{triangle2}
    \begin{tikzcd}[column sep=small]
    \Omega_M^1 \arrow[dr, ""] \arrow[rr, ""] & & \Omega_{M/B}^1 \arrow[dl, ""] \\
       & \Omega^1_{M/B}(\log F).  & 
\end{tikzcd}
\end{equation}
Here $\Omega_M^1 \to \Omega^1_{M/B}$ and $\Omega^1_{M/B} \to \Omega^1_{M/B}(\log F)$ are the natural maps, and their composition recovers the map $\Omega_{M}^1 \to \Omega_{M/B}^1(\log F)$ of (\ref{triangle1}). The octahedral axiom yields that the cones associated with the three maps of (\ref{triangle2}) form an exact triangle
\[
\pi_M^*\Omega^1_B \to \CK_0 \to \BC_x[-1] \xrightarrow{+1}
\]
where $\BC_x \simeq\mathrm{cone}\left( \Omega_{M/B}^1 \to \Omega_{M/B}^1(\log F)\right)$. Consequently, we have
\[
\CH^1(\CK_0) \simeq \BC_x
\]
which contradicts (\ref{contra}). 
\end{proof}

The extension class
\[
[\widetilde{\phi}_{\mathrm{FM}}(\Omega_M^1) ] \in \mathrm{Ext}_M^1(0_{B*}\Omega_B^{1\vee}, \widetilde{\CW}_0^{\mathrm{FM}} ) 
\]
is sent naturally to a class
\begin{equation}\label{Ext1}
\widetilde{\rho}([\widetilde{\phi}_{\mathrm{FM}}(\Omega_M^1) ] ) \in  \mathrm{Ext}_M^1(0_{B*}\Omega_B^{1\vee}, 0_{B*}\Omega^1_B(D) )
\end{equation}
via the natural morphism induced by Proposition \ref{prop6.4}:
\[
\widetilde{\rho}: \mathrm{Ext}_M^1(0_{B*}\Omega_B^{1\vee}, \widetilde{\CW}_0^{\mathrm{FM}} )  \to \mathrm{Ext}_M^1(0_{B*}\Omega_B^{1\vee}, 0_{B*}\Omega^1_B(D) ). 
\]
Applying the functor $\widetilde{\phi}_{\mathrm{FM}}$ to the triangle (\ref{triangle1}), we see that the class (\ref{Ext1}) represents the extension
\[
 \widetilde{\phi}_{\mathrm{FM}}(\pi_M^*\Omega^1_B(D)) \to  \widetilde{\phi}_{\mathrm{FM}}(\Omega^1_{M}(\log F)) \to \widetilde{\phi}_{\mathrm{FM}}(\Omega^1_{M/B}(\log F)) \xrightarrow{+1},
 \]
which is the exact triangle obtained by applying $\widetilde{\phi}_{\mathrm{FM}}$ to the exact sequence (\ref{6.2_0}). 

We note that both sheaves $0_{B*}\Omega_B^{1\vee}$ and $0_{B*}\Omega^1_B(D)$ are supported on the $0$-section. Therefore we have a natural isomorphism
\[
\mathrm{Ext}_M^1(0_{B*}\Omega_B^{1\vee}, 0_{B*}\Omega^1_B(D) ) = \mathrm{Ext}_{\hat{B}}^1(0_{B*}\Omega_B^{1\vee}, 0_{B*}\Omega^1_B(D) ).
\]
An identical argument as in Section \ref{Sec5.3} yields a splitting
\[
\mathrm{Ext}_{\hat{B}}^1(0_{B*}\Omega_B^{1\vee}, 0_{B*}\Omega^1_B(D) ) = \mathrm{Ext}_{{B}}^1(\Omega_B^{1\vee}, 0_{B*}\Omega_B(D)) \oplus H^0(B, (\Omega^1_B)^{\otimes 3}(D));
\]
this is the counter-part of (\ref{splitting_HM}) on the Fourier--Mukai side. Furthermore, the class (\ref{Ext1}) is~$(0, [\overline{\nabla}])$ with $[\overline{\nabla}]$ given by (\ref{cubic2}).

\subsection{Proof of Theorem \ref{thm6.1}}
We fix an identification of the formal neighborhoods of the $0$-sections in $T^*B$ and $M$ respectively using Proposition \ref{prop1.4}; we denote them by $\hat{B}$ uniformly.

On the Hodge module side, by Section \ref{Sec6.4} $\mathrm{gr}(j_{!*}\CV)$ is a good admissible sheaf on $T^*B$. Therefore its restriction 
\[
\mathrm{gr}(j_{!*}\CV)|_{\hat{B}}
\]
is a good admissible sheaf on $\hat{B}$.

On the Fourier--Mukai side, Propositions \ref{prop6.4} and \ref{prop6.5} imply that
\[
\phi_{\mathrm{FM}}(\Omega_M^1)|_{\hat{B}} = \widetilde{\phi}_{\mathrm{FM}}(\Omega_M^1)|_{\hat{B}}
\]
is a good admissible sheaf on $\hat{B}$. 

Furthermore, by the discussion at the ends of Sections \ref{Sec6.4} and \ref{Sec6.5} respectively, we have
\[
\rho_{\mathrm{HM}}([ \mathrm{gr}(j_{!*}\CV)|_{\hat{B}} ]) =  \rho_{\mathrm{FM}}([ \phi_{\mathrm{FM}}(\Omega_M^1)|_{\hat{B}} ]) = (0,[\overline{\nabla}])\in \mathrm{Ext}_{\hat{B}}^1(0_{B*}\Omega_B^{1\vee}, 0_{B*}\Omega^1_B(D)). \]
Finally we apply Proposition \ref{prop6.2} and conclude that the two good admissible sheaves obtained from the Hodge module and the Fourier--Mukai transform are isomorphic. This completes the proof. \qed

\subsection{Remarks on the cuspidal case}\label{sec6.7}
In fact the same strategy also applies to cuspidal fibers.\footnote{As we work with an elliptic fibration with integral fibers, nodes and cusps are the only possible singularities.} Here we briefly sketch the key steps and leave the details to the interested reader.

From now on we allow the elliptic fibration $\pi_M : M \to B$ to have nodal fibers $F_1, \ldots, F_m$ over $p_1, \ldots, p_m \in B$, and cuspidal fibers $G_1, \ldots, G_l$ over $q_1, \ldots, q_l \in B$. The rest of the assumptions on $\pi_M$ remains the same as in Section \ref{sec:6.1}, including $B$ non-proper. We set
\[
D := \sum_{i = 1}^mp_i \subset B, \quad C := \sum_{i = 1}^lq_i \subset B, \quad F := \sum_{i = 1}^mF_i \subset M, \quad G := \sum_{i = 1}^lG_i \subset M.
\]

In the presence of cusps, the notion of an \emph{admissible sheaf} $\CA \in \mathrm{Coh}(T^*B)$ should be altered to a $4$-step filtration
\[
\CW^{\CA, \mathrm{red}}_{-1} \subset \CW_{-1}^\CA \subset \CW_{0}^\CA \subset \CA
\]
which satisfies
\begin{enumerate}
    \item[(a)] $\CW^{\CA, \mathrm{red}}_{-1} \simeq \oplus_{i=1}^m \CO_{F_i} \oplus \oplus_{i = 1}^l \CO_{G_i}$,
    \item[(b)] $\CW^\CA_{-1}/\CW^{\CA, \mathrm{red}}_{-1} \simeq \oplus_{i = 1}^l \CO_{G_i}$,
    \item[(c)] $\CW^\CA_0/ \CW^\CA_{-1} \simeq 0_{B*}\Omega^1_B(D + 2C)$, and
    \item[(d)] $\CA/\CW^\CA_0 \simeq 0_{B*}\Omega_B^{1\vee}$.
\end{enumerate}
An admissible sheaf $\CA$ is \emph{good} if the following hold.
\begin{enumerate}
\item[(i)] The extension class
\[
[\CW^\CA_{-1}] \in \mathrm{Ext}^1_{T^*B}\left(\bigoplus_{i = 1}^l \CO_{G_i}, \bigoplus_{i=1}^m \CO_{F_i} \oplus \bigoplus_{i = 1}^l \CO_{G_i}\right) = \bigoplus_{i = 1}^{l}\mathrm{Ext}^1_{T^*B}(\CO_{G_i}, \CO_{G_i})
\]
has a nonzero summand in each $\mathrm{Ext}^1_{T^*B}(\CO_{G_i}, \CO_{G_i}) \simeq \BC$. Up to scaling $\CO_{G_i}$ we may assume that each summand is $1$. In particular, the sheaf $\CW^\CA_{-1}$ is of the form
\[
\CW^\CA_{-1} \simeq \bigoplus_{i=1}^m \CO_{F_i} \oplus \bigoplus_{i = 1}^l \CO_{\mathbf{G}_i}
\]
where $\mathbf{G}_i \subset T^*B$ is a thickening of $G_i$.
\item[(ii)] The extension class
\begin{multline*}
[\CW^\CA_{0}] \in \mathrm{Ext}^1_{T^*B}(0_{B*}\Omega^1_B(D + 2C), \CW^\CA_{-1}) \\
= \bigoplus_{i = 1}^m \mathrm{Ext}^1_{T^*B}(0_{B*}\Omega^1_B(D + 2C), \CO_{F_i}) \oplus \bigoplus_{i = 1}^l\mathrm{Ext}^1_{T^*B}(0_{B*}\Omega^1_B(D + 2C), \CO_{\mathbf{G}_i})
\end{multline*}
has a nonzero summand in each $\mathrm{Ext}^1_{T^*B}(0_{B*}\Omega^1_B(D + 2C), \CO_{F_i}) \simeq \BC$ as well as in each~$\mathrm{Ext}^1_{T^*B}(0_{B*}\Omega^1_B(D + 2C), \CO_{\mathbf{G}_i}) \simeq \BC^2$. Up to an automorphism of $\CO_{F_i}$ and $\CO_{\mathbf{G}_i}$ we may assume that the summands are either $1$ or $(1, 0)$.
\end{enumerate}
Further, the short exact sequence
\[
0 \to \CW^\CA_{-1} \to \CW^\CA_{0} \to 0_{B*}\Omega_B^1(D + 2C) \to 0
\]
induces a natural morphism
\[
\rho_\CA : \mathrm{Ext}^1_{T^*B}(0_{B*}\Omega_B^{1\vee}, \CW^\CA_{0}) \to \mathrm{Ext}^1_{T^*B}(0_{B*}\Omega_B^{1\vee}, 0_{B*}\Omega_B^1(D + 2C)),
\]
sending the extension class $[\CA] \in \mathrm{Ext}^1_{T^*B}(0_{B*}\Omega_B^{1\vee}, \CW^\CA_{0})$ to $\rho_\CA([\CA])$.

A version of Proposition \ref{prop6.2} states that over a non-proper curve $B$, two good admissible sheaves $\CA, \CA' \in \mathrm{Coh}(T^*B)$ are isomorphic if the classes $\rho_\CA([\CA])$ and $\rho_{\CA'}([\CA'])$ coincide:
\[
\rho_\CA([\CA]) =  \rho_{\CA'}([\CA']) \in \mathrm{Ext}^1_{T^*B}(0_{B*}\Omega_B^{1\vee}, 0_{B*}\Omega^1_B(D + 2C)).
\]

On the Hodge module side we have $P_1 = j_{!*}\CV$ as in Section \ref{sec:6.1}. The fact that $\mathrm{gr}(j_{!*}\CV)$ is a good admissible sheaf on $T^*B$ is again proven by an explicit calculation using the local monodromy $(\begin{smallmatrix} 1&1\\-1&0 \end{smallmatrix})$ around each $q_i \in B$ corresponding to the cuspidal fiber $G_i$. Note that however, this time one should distinguish the two canonical extensions with respect to $[0, 1)$ and $(-1, 0]$ (which differ by a twist by $\CO_B(C)$) and use the latter for $\overline{\CV}$ in \eqref{eq:ic}. The end results, deduced from formulas in \cite[Theorem 2.6]{K} and \cite{Ma3}, are
\begin{gather*}
\mathrm{gr}_{-1}(j_{!*}\CV) \simeq \Omega_B^{1\vee}, \quad \mathrm{gr}_{0}(j_{!*}\CV) \simeq \Omega^1_B(D + 2C),\\
\mathrm{gr}_{k}(j_{!*}\CV) \simeq \bigoplus_{i=1}^m \BC_{p_i} \oplus \bigoplus_{i = 1}^l \CO_{\mathbf{q}_i}, \quad k>0,
\end{gather*}
where $\mathbf{q}_i$ is the length $2$ fat point supported on $q_i$. The nontrivial Higgs fields \eqref{Higgs_fields} are given by a cubic form
\begin{equation} \label{eq:cubiccusp}
[\overline{\nabla}] \in H^0(B, (\Omega^1_B)^{\otimes 3}(D + 2C))
\end{equation}
for $k = -1$, the residue map for $k = 0$, and the identity maps for all $k > 0$. As in the nodal case we set $\CW_i^{\mathrm{HM}} \in \mathrm{Coh}(T^*B)$ to be
\[
\bigoplus_{k \geq -i}\mathrm{gr}_k(j_{!*}\CV)
\]
endowed with the restricted Higgs field, and $\CW_{-1}^{\mathrm{HM}, \mathrm{red}}$ is the obvious subsheaf of $\CW_{-1}^{\mathrm{HM}}$. It is straightforward to check that $\mathrm{gr}(j_{!*}\CV)$ is a good admissible sheaf on $T^*B$.

The Fourier--Mukai side is treated with the help of a log resolution. Let $f: \widetilde{M} \to M$ be the resolution obtained by blowing up each cusp $y_i \in G_i$ three times. We write $\pi_{\widetilde{M}}: \widetilde{M} \to B$ for the composition $\pi_M \circ f$. Let $E_{1, i}, E_{2, i}, E_{3, i} \subset \widetilde{M}$ be the three (strict transforms of) exceptional divisors associated with $y_i$. For $k = 1, 2, 3$, we set $E_k := \sum_{i = 1}^lE_{k, i} \subset \widetilde{M}$. We also set
\[
\mathbf{E} = E_1 + 2E_2 + 5E_3 \subset \widetilde{M}.
\]

We consider the following triangle of morphisms
\begin{equation} \label{eq:fancytriangle}
\begin{tikzcd}[column sep=small]
\Omega_M^1 \arrow[dr, ""] \arrow[rr, ""] & & Rf_*\left(\Omega^1_{\widetilde{M}}(\log f^{-1}(F + G))\otimes \CO_{\widetilde{M}}(\mathbf{E})\right) \arrow[dl, ""] \\
& Rf_*\left(\Omega^1_{\widetilde{M}/B}(\log f^{-1}(F + G))\otimes \CO_{\widetilde{M}}(\mathbf{E})\right). & 
\end{tikzcd}
\end{equation}
We define
\begin{gather*}
\CK_{-1}^{\mathrm{red}} := \mathrm{cone}\left(\Omega_M^1 \to Rf_*\left(\Omega^1_{\widetilde{M}}(\log f^{-1}(F + G))\otimes \CO_{\widetilde{M}}(\mathbf{E})\right)\right)[-1], \\
\CK_0 := \mathrm{cone}\left(\Omega_M^1 \to Rf_*\left(\Omega^1_{\widetilde{M}/B}(\log f^{-1}(F + G))\otimes \CO_{\widetilde{M}}(\mathbf{E})\right)\right)[-1].
\end{gather*}
By the octahedral axiom, the cones associated with the three maps of \eqref{eq:fancytriangle} form an exact triangle
\begin{equation} \label{eq:k-1k0}
\CK_{-1}^{\mathrm{red}} \to \CK_0 \to \pi_M^*\Omega^1_B(D + C) \otimes Rf_*\CO_{\widetilde{M}}(\mathbf{E}) \xrightarrow{+1}.
\end{equation}
Moreover, the natural inclusion $\mathbf{E} \subset \pi_{\widetilde{M}}^{-1}(C)$ induces a morphism
\begin{equation} \label{eq:inclusion}
\pi_M^*\Omega^1_B(D + C) \otimes Rf_*\CO_{\widetilde{M}}(\mathbf{E}) \to \pi_M^*\Omega^1_B(D + 2C).
\end{equation}
We define
\[
\CK_{-1} := \mathrm{cone}\left(\CK_0 \to \pi_M^*\Omega^1_B(D + 2C)\right)[-1]
\]
where the arrow is obtained by composing \eqref{eq:k-1k0} and \eqref{eq:inclusion}. Finally we set
\[
\CW^{\mathrm{FM}, \mathrm{red}}_{-1} := \phi_{\mathrm{FM}}(\CK_{-1}^{\mathrm{red}}), \quad \CW^{\mathrm{FM}}_{-1} := \phi_{\mathrm{FM}}(\CK_{-1}), \quad \CW^{\mathrm{FM}}_{0} := \phi_{\mathrm{FM}}(\CK_{0}).
\]

It remains to check that all three terms above are sheaves concentrated in degree $0$ which are part of a filtration
\[
\CW^{\mathrm{FM}, \mathrm{red}}_{-1} \subset \CW^{\mathrm{FM}}_{-1} \subset \CW^{\mathrm{FM}}_{0} \subset \phi_{\mathrm{FM}}(\Omega_M^1),
\]
and that $\phi_{\mathrm{FM}}(\Omega_M^1)|_{\hat{B}}$ is indeed a good admissible sheaf on $\hat{B}$. For example, by \cite[2.10]{K} we have
\[
\Omega^1_{\widetilde{M}/B}(\log f^{-1}(F + G))\otimes \CO_{\widetilde{M}}(\mathbf{E}) \simeq \omega_{\widetilde{M}/B}
\]
so that under the symplectic form $\sigma$ of $M$ there is an isomorphism
\[
Rf_*\left(\Omega^1_{\widetilde{M}/B}(\log f^{-1}(F + G))\otimes \CO_{\widetilde{M}}(\mathbf{E})\right) \simeq \pi_B^*\Omega_B^{1\vee}.
\]
Comparing with the definition of $\CK_0$, we find an exact triangle
\[
\CK_0 \to \Omega^1_M \to \pi_B^*\Omega_B^{1\vee} \xrightarrow{+1}
\]
whose Fourier--Mukai image is the expected short exact sequence
\[
0 \to \CW^{\mathrm{FM}}_{0} \to \phi_{\mathrm{FM}}(\Omega_M^1) \to 0_{B*}\Omega_B^{1\vee} \to 0.
\]
One also uses the exact triangle
\begin{multline*}
\pi_M^*\Omega^1_B(D + C) \otimes Rf_*\CO_{\widetilde{M}}(\mathbf{E}) \to Rf_*\left(\Omega^1_{\widetilde{M}}(\log f^{-1}(F + G))\otimes \CO_{\widetilde{M}}(\mathbf{E})\right) \\
\to Rf_*\left(\Omega^1_{\widetilde{M}/B}(\log f^{-1}(F + G))\otimes \CO_{\widetilde{M}}(\mathbf{E})\right) \xrightarrow{+1}
\end{multline*}
and the log Katz--Oda theorem \cite{Katz} to relate the extension class $\rho_{\mathrm{FM}}([\phi_{\mathrm{FM}}(\Omega_M^1)|_{\hat{B}}])$ to the same cubic form $[\overline{\nabla}]$ as in \eqref{eq:cubiccusp}, and to conclude that
\[
\rho_{\mathrm{HM}}([ \mathrm{gr}(j_{!*}\CV)|_{\hat{B}} ]) = \rho_{\mathrm{FM}}([\phi_{\mathrm{FM}}(\Omega_M^1)|_{\hat{B}}])
\]
as in the nodal case.

\end{document}